\documentclass[11pt]{amsart}
\usepackage[T1]{fontenc}
\usepackage{amsmath, amssymb, amsfonts}
\usepackage{dsfont}
\usepackage{natbib}
\usepackage{graphicx}
\usepackage{slashbox}
\usepackage{color}
\usepackage{aeguill}
\usepackage{multicol}
\usepackage{multirow}
\usepackage{fancyhdr}
\usepackage{pstricks}
\newtheorem{prop}{Proposition}
\newtheorem{lem}{Lemma}
\newtheorem*{lemn}{Lemma}
\newtheorem{thm}{Theorem}
\newtheorem*{thmn}{Theorem}
\newtheorem{cor}{Corollary}
\usepackage{bm}
\theoremstyle{remark}
\newtheorem{Example}{Example}

\newtheorem*{examplen}{Example 1 continued}
\newtheorem{remark}{Remark}

\newcommand{\norm}[1]{\|#1\|}

\DeclareMathOperator*{\argmin}{argmin}
\DeclareMathOperator*{\spec}{Sp}

\def\E{\mathbb{E}}
\def\pen{{\mathrm{pen}}}
\def\1{\mathds{1}}

\begin{document}
\title[Estimation of the conditional intensity]{Adaptive estimation of
  the conditional intensity of marker-dependent counting processes}
\author[F. Comte, S. Ga\"iffas \& A. Guilloux]{F. Comte$^{(1)}$, S. Ga\"iffas$^{(2)}$ \& A. Guilloux$^{(3)}$}
\thanks{\noindent $^{(1)}$ MAP5, University
  Paris Descartes, France. email: fabienne.comte@parisdescartes.fr,\\
  $^{(2)}$ LSTA University Pierre et Marie Curie, France. email:
  stephane.gaiffas@upmc.fr,\\ $^{(3)}$ LSTA University Pierre et Marie
  Curie, France. email: agathe.guilloux@upmc.fr}

\begin{abstract}
  We propose in this work an original estimator of the conditional
  intensity of a marker-dependent counting process, that is, a
  counting process with covariates.  We use model selection methods
  and provide a non asymptotic bound for the risk of our estimator on
  a compact set. We show that our estimator reaches automatically a
  convergence rate over a functional class with a given (unknown)
  anisotropic regularity. Then, we prove a lower bound which
  establishes that this rate is optimal. Lastly, we provide a short
  illustration of the way the estimator works in the context of
  conditional hazard estimation.
\end{abstract}
\maketitle

\begin{center} \today \end{center}

\noindent {\it AMS (2000) subject classification}. 62N02, 62G05.

Conditional hazard function.
\section{Introduction}
\label{sec:intro}

As counting processes can model a great diversity of observations,
especially in medicine, actuarial science or economics, their
statistical inference has received a continuous attention since half a
century - see \cite{ABGK} for the most detailed presentation on the
subject. In this paper, we propose a new strategy, based on model
selection, for the inference for counting processes in presence of
covariates. The model considered can be described as follows.

Let $(\Omega, \mathcal F, \mathbb P)$ be a probability space and
$(\mathcal F_t)_{t \geq 0}$ a filtration satisfying the usual
conditions. Let $N$ be a marker-dependent counting process, with
compensator $\Lambda$ with respect to $(\mathcal F_t)_{t \geq 0}$,
such that $N-\Lambda=M$, where $M$ is a $(\mathcal F_t)_{t \geq
  0}$-martingale. We assume that $N$ is a marker-dependent counting
process satisfying the Aalen multiplicative intensity model in the
sense that :
\begin{eqnarray}\label{eq:aalen}
\Lambda(t)=\int_0^t \alpha(X,z) Y(z)dz, \text{ for all } t \geq0
\end{eqnarray}
where  $X$ is a vector of covariates in $\mathbb R^d$ which is
$\mathcal F_0$-measurable, the process $Y$ is nonnegative and
predictable and $\alpha$ is an unknown deterministic function called
intensity.

The purpose of this paper is to estimate the intensity function
$\alpha$ on the basis of the observation of a $n$-sample
$(X_i,N^i(z),Y^i(z), z\leq \tau)$ for $i=1,\dots,n$, where $\tau <
+\infty$.

There are many examples, crucial in practice, which fulfill this
model. For the seek of conciseness, we restrict our presentation to
the three following ones.

\begin{Example}[Regression model for right-censored data]\label{regcens}
  Let $T$ be a nonnegative random variable (r.v.) and $X$ a vector of
  covariates in ${\mathbb R}^d$, with respective cumulative
  distribution functions (c.d.f.) $F_T$ and $F_X$. We consider in
  addition that $T$ can be censored. We introduce the nonnegative
  r.v. $C$, with c.d.f. $G$, such that the observable r.v.  are $Z=T
  \wedge C$, $\delta=\1(T \leq C)$ and $X$. We assume that:
\begin{eqnarray*}
(\mathcal C): \;T  \text{ and } C \text{ are independent conditionally to } X.
\end{eqnarray*}
In this case, the processes to consider (see e.g. \cite{ABGK}) are
given, for $i=1,\dots,n$ and $z \geq 0$, by:
\begin{eqnarray*}
  N^i (z) = \1(Z_i \leq z, \delta_i=1) \text{ and } Y^i (z) =
  \1( Z_i \geq z ).
\end{eqnarray*}
The unknown intensity function $\alpha$ to be estimated is the
conditional hazard rate of the r.v. $T$ given $X=x$ defined, for all
$z>0$ by:
\begin{eqnarray*}
\alpha(x, z)=\alpha_{T|X}(x,z) = \frac{ f_{T|X}(x,z)}{1-F_{T|X}(x,z)},
\end{eqnarray*}
where $f_{T|X}$ and $F_{T|X}$ are respectively the conditional
probability density function (p.d.f.) and the conditional c.d.f. of
$Y$ given $X$.

Nonparametric estimation of the hazard rate in presence of
covariates was initiated by \cite{Beran}. \cite{STU}, \cite{DAB87},
\cite{McKeague} and \cite{LD} extended his results. Many authors
have considered semiparametric estimation of the hazard rate,
beginning with \cite{Cox}, see \cite{ABGK} for a review of the
enormous literature on semiparametric models. We refer to
\cite{Huang} and \cite{LNVdG} for some recent developments.

As  far as we know, adaptive nonparametric estimation for censored
data in presence of covariates has only been considered in
\cite{BCL}, who constructed an optimal adaptive estimator of the
conditional density.\end{Example}

\begin{Example}[Cox processes]
  Let $\eta^i$, for $i=1,\dots,n$, be a Cox process (see~\cite{karr})
  on $\mathbb R_+$ with random mean-measure $\Lambda^i$ given by :
\begin{eqnarray*}
\Lambda^i(t) = \int_0^t \alpha(X_i,z) dz,
\end{eqnarray*}
where $X_i$ is a vector of covariates in $\mathbb R^d$. In this
context the predictable process $Y$ of Equation~\eqref{eq:aalen}
constantly equals~1. As a consequence, these processes can be seen as
generalizations of nonhomogeneous Poisson processes on $\mathbb R_+$
with random intensities. This is a particular case of longitudinal
data, see e.g. Example VII.2.15 in \cite{ABGK}. The nonparametric
estimation of the intensity of Poisson processes without covariates
has been considered in several papers. We refer to \cite{Patricia1}
and \cite{baraud-2008} for the adaptive estimation of the intensity of
nonhomogeneous Poisson processes in general spaces.
\end{Example}

\begin{Example}[Regression model for transition intensities of Markov
  processes]
  Consider a $n$-sample of nonhomogeneous time-continuous Markov
  processes $P^1,\dots,P^n$ with finite state space $\{1,\dots,k\}$
  and denote by $\alpha_{jl}$ the transition intensity from state $j$
  to state $l$. For individual $i$ with covariate $X_i$, let
  $N^i_{jl}(t)$ be the number of observed direct transitions from $j$
  to $l$ before time $t$ (we allow the possibility of right-censoring
  for example). Conditionally on the initial state, the counting
  process $N^i_{jl}$ verifies the following Aalen multiplicative
  intensity model:
  \begin{eqnarray*}
    N^i_{jl}(t) = \int_0^t \alpha_{jl}(X_i,z)Y^i_j(z)dz+M^i(t) \text{
      for all } t \geq 0,
  \end{eqnarray*}
  where $Y^i_j(t)=\1\{P^i(t-)=j\}$ for all $t \geq 0$, see \cite{ABGK}
  or \cite{Jacobsen}. This setting is discussed in \cite{ABGK}, see
  Example VII.11 on mortality and nephropathy for insulin dependent
  diabetics.\end{Example}

We finally cite three papers, where different strategies for the
estimation of the intensity of counting processes is considered,
gathering as a consequence all the previous examples, but in none of
them the presence of covariates was considered. \cite{Ramlau} proposed
a kernel-type estimator, \cite{Gregoire} studied cross-validation for
these estimators. More recently, \cite{Patricia2} considered adaptive
estimation by model selection.

Our aim in this work is to provide an optimal adaptive nonparametric
estimator of the conditional intensity. Our estimation procedure
involves the minimization of a so-called contrast. To achieve that
purpose, we proceed as follows. In Section \ref{procedure}, we
describe the estimation procedure: we explain how the contrast is
built, on which collections of spaces the estimators are defined and
how the relevant space is selected via a data driven penalized
criterion. In Section~\ref{results}, we state an oracle inequality
for our estimator (see Theorem~\ref{main}), a resulting upper bound
(see Corollary~\ref{coromain}) and a lower bound (see
Theorem~\ref{thm:lower_bound}), the latter asserts the optimality in
the minimax sense. An auxiliary estimation of the density of the
reference measure is also studied. The examples of
Section~\ref{simu} are taken in the setting of
Example~\ref{regcens}, in order to provide a short illustration of
the practical properties of our estimator. Lastly, proofs are
gathered in Sections
\ref{proofs}-\ref{sec:deviation}-\ref{auxiliary}. We mention that
the deviation inequalities proved in Section \ref{sec:deviation} may
be of intrinsic interest.

\begin{remark}
  \label{rem:aniso}
  An inherent remark about this model is that there is no reason for
  the conditional intensity $\alpha(x, z)$ to have the same behavior
  with respect to the $z$ (time) and $x$ (covariates) variables. This
  is the reason why it is mandatory in our purely nonparametric
  setting to consider anisotropic regularity for $\alpha$. Think for
  instance of the very popular case of proportional hazards Cox model,
  see~\cite{Cox}, it is assumed that $\alpha(x, z) = \alpha_0(z)
  \exp(\beta^\top x)$ for some unknown function $\alpha_0$ and unknown
  vector $\beta \in \mathbb R^d$. Of course, in this model, the
  smoothness in the $x$ direction is higher than in the $z$ direction.
\end{remark}

For the sake of simplicity, we will assume in the following that the
covariate $X$ is one-dimensional. Similar procedures and results for
multivariate covariates are an almost effortless extension, as
discussed in Remark~\ref{remRd}.

\section{Description of the procedure}
\label{procedure}

Our estimation procedure involves the minimization of a contrast. This
contrast is tuned to the problem considered in this paper, as
explained in the next section.

\subsection{Definition of the contrast}

Let $A = A_1 \times A_2$ be a compact set on $\mathbb R \times \mathbb
R_+$ on which the function $\alpha$ will be estimated. Without loss of
generality, we set $A=[0,1]\times[0,1]$, and in particular
$\tau=1$. Let $h$ be a function in $(L^2\cap L^\infty)(A)$. Define the
contrast function:
\begin{eqnarray}\label{eqn:contrast}
\gamma_n(h)= \frac 1n \sum_{i=1}^n \int_0^1 h^2(X_i,z) Y^i(z) dz- \frac
2n  \sum_{i=1}^n \int_0^1 h(X_i,z)dN^i(z).
\end{eqnarray}
This contrast is of least-squares type adapted to the problem
considered here. Since each $N^i$ admits a Doob-Meyer decomposition
($N^i=\Lambda^i+M^i$), we have:
\begin{eqnarray*}
\gamma_n(h)=\frac 1n \sum_{i=1}^n \int_0^1 h^2(X_i,z)Y^i(z) dz- \frac
2n  \sum_{i=1}^n \int_0^1 h(X_i,z)d\Lambda^i(z)- \frac
2n  \sum_{i=1}^n \int_0^1 h(X_i,z)dM^i(z), \end{eqnarray*}
so that:
\begin{eqnarray*}
  \mathbb E\big(\gamma_n(h)\big)=\mathbb E \big(  \int_0^1 h^2(X,z)Y(z) dz\big)- \mathbb E\big(
  2 \int_0^1 h(X,z)d\Lambda(z)). \end{eqnarray*}
Let $F_X$ denote the c.d.f. of the covariate $X$ and  $\|\cdot\|_{\mu}$ the norm defined by:
\begin{eqnarray*}
\|h\|^2_{\mu}&:=& \mathbb E \big( \int_0^1 h^2(X,z)Y(z) dz \big)=\iint_A h^2(x,z)d\mu(x,z), \end{eqnarray*} where
$d\mu(x,z):= {\mathbb E}(Y(z)|X=x)F_X(dx)dz.$
By the Aalen multiplicative intensity model, see Equation \eqref{eq:aalen}, we get:
\begin{eqnarray*}
\mathbb E\big(\gamma_n(h)\big)=\|h\|_{\mu}^2 -2 \iint
h(x,z)\alpha(x,z) {\mathbb E}(Y(z)|X=x)F_X(dx)dz= \|h-\alpha\|_{\mu}^2
-\|\alpha\|^2_{\mu}.
\end{eqnarray*}
This explains why minimizing $\gamma_n(\cdot)$ over an appropriate set
of functions described  below, is a relevant strategy to estimate $\alpha$.

\begin{examplen}
  In the particular case of regression for right-censored data, the
  conditional hazard function is estimated and the contrast function
  has the following form:
\begin{eqnarray*}
  \gamma_n(h)&=& \frac 1n \sum_{i=1}^n \int_0^1 h^2(X_i,z)\1(Z_i\geq z) dz- \frac 2n  \sum_{i=1}^n \delta_i h(X_i,Z_i).
\end{eqnarray*}
We have in addition an explicit formula for $d\mu(x,z)$:
$$d\mu(x,z)=(1-L_{Z|X}(z,x)) F_X(dx)dz,$$ where $$1-L_{Z|X}(z,x):=
\mathbb P( Z \geq z | X = x ) = (1 - F_{T|X}(x,z))(1-G_{C|X}(x,z))$$ and $G_{C|X}$ is the
conditional c.d.f. of $C$ given $X$.

\begin{remark}
  In our setting, it is possible to let the censoring depend on the
  covariates, as in \cite{DAB89} or, more recently
  \cite{Heuch}. Assumption $(\mathcal C)$ above is weaker than the
  assumption: $T$ and $C$ are independent and $\mathbb P(T \leq
  C|X,Y)=\mathbb P(T \leq C|Y)$ in \cite{STUT}.
\end{remark}
\end{examplen}

\subsection{Assumptions and notations}

Before defining the estimation procedure, we need to introduce some
assumptions and notations. Define the norms
\begin{equation*}
\|h\|^2 := \iint h^2(x,z)dxdz, \| h \|_A^2 := \iint_A h^2(x,z) dx dz
\text{ and } \|h\|_{\infty,A} := \sup_{(x,z)\in A} |h(x,z)|,
\end{equation*}
and assume that the following holds:
\begin{itemize}
\item ($ \mathcal A 1$) The covariates $X_i$ admit a p.d.f. $f_X$ such
  that $\sup_{A_2} |f_X | < +\infty$.
\end{itemize}
Assumption $(\mathcal A1)$ implies that $\mu$ admits a density
w.r.t. the Lebesgue measure. We denote by $f$ this density:
\begin{equation}
\label{eq:def_f}
d \mu(x, z) = f(x,z) dx dz \text{ where } f(x,z)= {\mathbb E}
(Y(z)|X=x) f_X(x).
\end{equation}
We also assume:
\begin{itemize}
\item ($\mathcal A 2$) There exists $f_0>0$, such that $\forall
(x,z)\in A_1\times A_2, \; f(x,z)\geq f_0$.
\item ($\mathcal A 3$) $\forall (x,z)\in A_1\times A_2, \;
\alpha(x,z)\leq \|\alpha\|_{\infty,A}<+\infty$.
\item $(\mathcal A4)$ $\forall i, \forall t,\; Y^i(t)\leq C_Y$ where
$C_Y$ is a known fixed constant.
\end{itemize}
Note that in the examples described in Section~\ref{sec:intro},
Assumption~$({\mathcal A}4)$ is clearly fulfilled with $C_Y=1$. We
will set $C_Y=1$ in the following for simplicity.

\subsection{Definition of the estimator}
\label{defestproj}

We use the usual model selection paradigm (see, for instance,
\cite{massart}): first minimize the contrast $\gamma_n(\cdot)$ over a
finite-dimensional function space $S_m$, then select the appropriate
space by penalization. We introduce a collection $\{S_m ,m\in
\mathcal{M}_n\}$ of projection spaces: $S_m$ is called a model and
${\mathcal M}_n$ is a set of multi-indexes (see the examples in
Section \ref{hypmo}). For each $m=(m_1,m_2)$, the space $S_m$ of
functions with support in $A=A_1\times A_2$ is defined by:
$$S_m = F_{m_1} \otimes H_{m_2} =
\Big\{ h, \quad h(x,z)=\sum_{j\in J_m}\sum_{k\in K_m} {a}_{j,k}^m
\varphi_j^m(x)\psi_k^m(z), \; a_{j,k}^m\in {\mathbb R} \Big\},$$ where
$F_{m_1}$ and $H_{m_2}$ are subspaces of $(L^2\cap
L^\infty)(\mathbb{R})$ respectively spanned by two orthonormal bases
$(\varphi_j^m)_{j\in J_m}$ with $|J_m|=D_{m_1}$ and $(\psi_k^m)_{k\in
  K_m}$ with $|K_m|=D_{m_2}$. For all $j$ and all $k$, the supports of
$\varphi_j^m$ and $\psi_k^m$ are respectively included in $A_1$ and
$A_2$. Here $j$ and $k$ are not necessarily integers, they can be
couples of integers, as in the case of a piecewise polynomial space,
see Section \ref{hypmo}.

\begin{remark}\label{remRd}
  From a theoretical point of view, we could consider that the
  covariates $X$ are in $\mathbb R^d$ and even that their density has
  an anisotropic regularity. For this end, we would have to consider
  models of the form $S_m = F_{m_1}\otimes H_{m_2}\otimes \dots
  \otimes H_{m_{d+1}}$. However, this would make the proofs more
  intricate. Notice also the convergence rate would be slower because
  of the curse of dimensionality.  For the sake of clarity, we
  deliberately restrict ourselves to $X \in \mathbb R$. \end{remark}

The first step would be to define $\hat\alpha_m=\argmin_{h\in S_m}
\gamma_n(h)$. To that end, let $h(x,y)=\sum_{j\in J_m}\sum_{k\in K_m}$ ${a}_{j,k}
\varphi_j^m(x)\psi_k^m(y)$ be a function in $S_m$. To compute
$\hat\alpha_m$, we have to solve:
$$  \forall j_0 \forall k_0, \quad
\frac{\partial \gamma_n(h)}{\partial {a}_{j_0,k_0}}=0
\Leftrightarrow G_m{A_m}= \Upsilon_m,$$ where ${A_m}$ denotes the
matrix $({a}_{j,k})_{j\in J_m, k\in K_m}$,
$$
G_m := \Big(\displaystyle
\frac{1}{n}\sum_{i=1}^n\varphi_j^m(X_i)\varphi_l^m(X_i)\int
\psi_k^m(z)\psi_p^m(z)Y^i(z) dz \Big) _{(j,k), (l,p) \in J_m\times
K_m} $$ and $$
\Upsilon_m := \Big(\displaystyle\cfrac{1}{n}\sum_{i=1}^n\varphi_j^m(X_i)
\int \psi_k^m(z)dN^i(z)\Big) _{j\in J_m, k \in K_m}. $$
Unfortunately $G_m$ may not be invertible. To overcome this problem,
we modify the definition of $\hat\alpha_m$ in the following way:
\begin{eqnarray}\label{alpha}
\hat \alpha_m := \Big\{\begin{array}{ll} \argmin_{h\in S_m}
    \gamma_n(h) & \mbox{ on } \hat \Gamma_m \\ 0 & \mbox{ on } \hat
    \Gamma_m^\complement\end{array}\Big.,
\end{eqnarray}
where
$$\hat\Gamma_m := \Big\{ \min{\rm Sp}(G_m)\geq   \max(\hat f_0/3,
n^{-1/2})\Big\}$$ where Sp$(G_m)$ denotes the spectrum of $G_m$
i.e. the set of the eigenvalues of the matrix $G_m$ (it is easy to see
that they are nonnegative). The estimator $\hat f_0$ of $f_0$ (the
minimum of the density $f$, see $({\mathcal A}2)$) is required to
fulfill the following assumption:
\begin{itemize}
\item ($\mathcal A 5$) For any integer $k\geq 1$, ${\mathbb
    P}(|\hat f_0-f_0|> f_0/2)\leq C_k/n^k$.
\end{itemize}
An estimator satisfying ($\mathcal A 5$) is defined in
Section~\ref{sec:estimation_de_f}. In fact, $k=7$ is enough for the
proofs. We refer the reader to the proof of Lemma~\ref{inclus}, see
Section~\ref{auxiliary}, for an explanation of the presence of
$n^{1/2}$ in the definition of $\hat\Gamma_m$. In practice, this
constraint is generally not used (the matrix is invertible, otherwise
another model is considered).

The final step is to select the relevant space via the penalized
criterion:
\begin{equation}\label{select}
\hat m=\argmin_{m \in {\mathcal M}_n} \Big(\gamma_n(\hat\alpha_m)+
  {\rm pen}(m)\Big),
\end{equation}
where ${\rm pen}(m)$ is defined in Theorem~\ref{main} below, see
Section~\ref{results}. Our estimator of $\alpha$ on $A$ is then
$\hat\alpha_{\hat m}$.

\subsection{Assumptions on the models and examples}\label{hypmo}

Let us introduce the following set of assumptions on the models $\{
S_m : m \in \mathcal M_n \}$, which are usual in model selection
techniques.

\begin{itemize}
\item $(\mathcal M1)$ For $i=1,2$, $\mathcal{D}_n^{(i)}:=\max_{m\in
  \mathcal{M}_n} D_{m_i}\leq n^{1/4}/ \sqrt{\log n}$.
\item $(\mathcal M2)$ There exist positive reals $\phi_1, \phi_2$ such
that, for all $u$ in $F_{m_1}$ and for all $v$ in $H_{m_2}$, we have
\begin{equation*}
  \sup_{x\in A_1}|u(x)|^2\leq \phi_1D_{m_1} \int_{A_1} u^2 \text{ and } \sup_{x\in
    A_2}|v(x)|^2\leq \phi_2D_{m_2}\int_{A_2} v^2.
\end{equation*}
By letting $\phi_0=\sqrt{\phi_1\phi_2}$, that leads to
\begin{equation}
  \forall h\in S_m \qquad \|h\|_{\infty,A}\leq
  \phi_0\sqrt{D_{m_1}D_{m_2}}\|h\|_A\label{M2}.
\end{equation}
\item $(\mathcal M3)$ Nesting condition:
  \begin{equation*}
    D_{m_1}\leq D_{m_1'}\Rightarrow F_{m_1}\subset F_{m_1'} \text{ and
    } D_{m_2} \leq D_{m_2'}\Rightarrow H_{m_2}\subset H_{m_2'}.
  \end{equation*}
  Moreover, there exists a global nesting space ${\mathcal S}_n$ in
  the collection, such that $\forall m\in {\mathcal M}_n, S_m\subset
  {\mathcal S}_n$ and dim$({\mathcal S}_n):=N_n\leq \sqrt{n/\log n}$.
\end{itemize}
Assumptions $(\mathcal M1)$--$(\mathcal M3)$ are not too
restrictive. Indeed, they are verified for the spaces $F_{m_1}$ (and
$H_{m_2}$) on $A_1 = [0, 1]$ spanned by the following bases (see
\cite{BBM}):

\begin{itemize}
\item $[T]$ Trigonometric basis: ${\rm
    span}(\varphi_0,\dots,\varphi_{m_1-1})$ with
  $\varphi_0=\1([0,1])$, $\varphi_{2j}(x)=\sqrt{2}$ $\cos(2\pi jx)$
  $\1([0,1])(x)$, $\varphi_{2j-1}(x)=$ $\sqrt{2}\sin(2\pi jx)
  \1([0,1])(x)$ for $j\geq 1$. For this model $D_{m_1}=m_1$ and
  $\phi_1=2$ hold.

\item $[DP]$ Regular piecewise polynomial basis: polynomials of degree
  $0,\dots,r$ (where $r$ is fixed) on each interval
  $[(l-1)/2^D,l/2^D[$ with $l=1,\dots,2^D$. In this case, we have
  $m_1=(D,r)$, $J_m=\{j=(l,d), \; 1\leq l\leq 2^D, 0\leq d\leq r\}$,
  $D_{m_1}=(r+1)2^D$ and $\phi_1=\sqrt{r+1}$.

\item $[W]$ Regular wavelet basis: ${\rm span}(\Psi_{lk},
  l=-1,\dots,m_1, k\in\Lambda(l))$ where $\Psi_{-1,k}$ is the
  translates of the father wavelet $\Psi_{-1}$ and
  $\Psi_{lk}(x)=2^{l/2}\Psi(2^lx-k)$ where $\Psi$ is the mother
  wavelet. We assume that the supports of the wavelets are included in
  $A_1$ and that $\Psi_{-1}$ belongs to the Sobolev space $W_2^r$, see
  \cite{HKPT}.

\item $[H]$ Histogram basis: for $A_1=[0,1]$, ${\rm
  span}(\varphi_1,\dots,\varphi_{2^{m_1}})$ with $\varphi_j
=2^{m_1/2}\1([(j-1)/2^{m_1},{j}/2^{m_1}[) $ for
$j=1,\dots,2^{m_1}$. Here $D_{m_1}=2^{m_1}$, $\phi_1=1$. Notice that $[H]$ is a particular case of both $[DP]$ and $[W]$.
\end{itemize}

\begin{remark}
  The first assumption prevents the dimension to be too large compared
  to the number of observations. We can lighten considerably this
  constraint for localized basis: for histogram basis, piecewise
  polynomial basis and wavelets, $(\mathcal M1)$ reduces to ${\mathcal
    D}_n^{(i)}\leq \sqrt{n/\log n}$.  Analogously in $(\mathcal M3)$,
  we would get $N_n\leq n/\log n$.  The condition $(\mathcal M2)$
  implies a useful link between the $L^2$ norm and the infinite norm.
  The third assumption $(\mathcal M3)$ implies in particular that
  $\forall m,m'\in {\mathcal M}_n$, $S_m+S_{m'}\subset {\mathcal
    S}_n$. This condition is useful for the chaining argument used in
  the proofs, see Section~\ref{sec:deviation}.
\end{remark}

\section{Main results}
\label{results}

\subsection{Oracle inequality}


For a function $h$ and a space $S$, let
$$d(h,S)=\inf_{g\in S}\|h-g\|=\inf_{g\in S}\Big(\iint
|h(x,y)-g(x,y)|^2dxdy\Big)^{1/2}.$$ The estimator $\hat \alpha_{\hat
m}$ where $\hat \alpha_m$ is given respectively by~\eqref{alpha} and
$\hat m$ is given by~\eqref{select} satisfies the following oracle
inequality.
\begin{thm}
  \label{main}
Let $(\mathcal A 1)$ -- $(\mathcal A 5)$ and $(\mathcal M1)$ --
$(\mathcal M3)$ hold. Define the following penalty\textup:
\begin{equation}
  \label{penalite}
  \pen(m) := K_0 (1+\| \alpha \|_{\infty,A})   \frac{D_{m_1}D_{m_2}}n,
\end{equation}
where $K_0$ is a numerical constant. We have
\begin{equation}
  \label{risk}
  \E(\|\alpha\1(A) - \hat \alpha_{\hat m} \|^2) \leq C \underset{m
    \in \mathcal{M}_n}{\inf}\{d^2(\alpha\1(A),S_m)
  +\pen(m)\}+\frac{C'}{n}
\end{equation}
where $C = C(f_0, \|f\|_{A,\infty})$ and $C'$ is a constant
depending on $\phi_1, \phi_2, \|\alpha\|_{\infty,A}, f_0.$
\end{thm}

The proof of Theorem~\ref{main} involves a deviation inequality for
the empirical process
\begin{equation*}
  \nu_n(h) :=  \frac1n \sum_{i=1}^n \int_0^1 h(X_i,z) d M^i(z),
\end{equation*}
where $M^i(t) = N^i(t) - \int_0^t \alpha(X_i, z) Y^i(z) dz$ are
martingales, see Section~\ref{sec:intro}, and a $L^2 - L^\infty$
chaining argument.

\begin{remark}
The penalty involves the unknown quantity
$\|\alpha\|_{\infty,A}$. This is a usual situation, and the
solution is to replace it by an estimator
$\|\hat\alpha_{m_n}\|_{\infty,A}$ where $\hat\alpha_{m_n}$ is an
estimator of the collection, chosen on a space $S_{m_n}$ which is
arbitrary, generally middle sized. Note that, by doing this, the
penalty function becomes random. For details, we refer to
\cite{LAC}, Theorem 2.2.
\end{remark}

\subsection{Upper bound for the rate}

From Theorem~\ref{main}, we can derive the rate of convergence of
$\hat \alpha_{\hat m}$ over anisotropic Besov spaces. We recall that
anisotropy is almost mandatory in this context, see
Remark~\ref{rem:aniso}. For that purpose, assume that $\alpha$
restricted to $A$ belongs to the anisotropic Besov space
$B_{2,\infty}^{\boldsymbol{\beta}}(A)$ on $A$ with regularity
$\boldsymbol{\beta}=(\beta_1,\beta_2)$. Let us recall the definition
of $B_{2,\infty}^{\boldsymbol{\beta}}(A)$. Let $\{ e_1, e_2 \}$ the
canonical basis of $\mathbb{R}^2$ and take $A_{h,i}^r := \{x\in
\mathbb{R}^2 ; x, x+he_i, \dots, x+rhe_i \in A\}$, for $i=1,2$.  For
$x \in A_{h,i}^r$, let
$$\Delta_{h,i}^rg(x) = \sum_{k=0}^r(-1)^{r-k}\binom{r}{k}g(x+khe_i)$$
be the $r$th difference operator with step $h$.  For $t>0$, the
directional moduli of smoothness are given by
$$\omega_{r_i,i}(g,t)=\underset{|h|\leq t}{\sup}\Big(\int_{A_{h,i}^{r_i}}
|\Delta_{h,i}^{r_i}g(x)|^2dx\Big)^{1/2}.$$ We say that $g$ is in
the Besov space $B_{2,\infty}^{\boldsymbol{\beta}}(A)$ if
$\sup_{t>0} \sum_{i=1}^2t^{-\beta_i}\omega_{r_i,i}(g,t)<\infty$ for
$r_i$ integers larger than $\beta_i$. More details concerning Besov
spaces can be found in \cite{Triebel}. The next corollary
shows that $\hat \alpha_{\hat m}$ adapts to the unknown anisotropic
smoothness of $\alpha$.

\begin{cor}
\label{coromain}
Assume that $\alpha$ restricted to $A$ belongs to the anisotropic
Besov space $B_{2,\infty}^{\boldsymbol{\beta}}(A)$ with regularity
$\boldsymbol{\beta}=(\beta_1,\beta_2)$ such that $\beta_1>1/2$ and
$\beta_2>1/2$. We consider the piecewise polynomial or wavelet
spaces described in Subsection~\ref{hypmo} \textup(with the
regularity $r$ of the polynomials and the wavelets larger than
$\beta_i-1$\textup). Then, under the assumptions of
Theorem~\ref{main}, we have
\begin{equation*}
  \E\| \alpha - \hat \alpha_{\hat m} \|_A^2 =
  O(n^{-\frac{2\bar \beta}{2\bar \beta + 2}}).
\end{equation*}
where $\bar\beta$ is the harmonic mean of $\beta_1$ and $\beta_2$
\textup(i.e. $2/\bar\beta=1/\beta_1+1/\beta_2$\textup).
\end{cor}

The rate of convergence achieved by $\hat \alpha_{\hat m}$ in
Corollary~\ref{coromain} is optimal in the minimax sense as proved in
Theorem~\ref{thm:lower_bound} below. For trigonometric spaces, the
result also holds, but for $\beta_1>3/2$ and $\beta_2>3/2$ (because of
$(\mathcal M1)$).

Moreover, assuming for example that $\beta_2>\beta_1$, one can see in
the proof of Corollary~\ref{coromain} that the estimator chooses a
space of dimension $D_{\hat m_2}=D_{\hat
  m_1}^{\beta_1/\beta_2}<D_{\hat m_1}$. This shows that the estimator
is adaptive with respect to the approximation space for each directional
regularity.

\subsection{Lower bound}

In the next Theorem, we prove that the rate $n^{-2 \bar \beta / (2\bar
\beta + 2)}$ is optimal over $B_{2, \infty}^{\boldsymbol{\beta}}(A)$
where we recall that $2 / \bar \beta = 1/\beta_1 + 1/\beta_2$. Since
the lower bound stated in Theorem~\ref{thm:lower_bound} is uniform
over $B_{2, \infty}^{\boldsymbol{\beta}}(A)$, we need to introduce the
ball
\begin{equation*}
B_{2, \infty}^{\boldsymbol{\beta}}(A, L) = \{ \alpha \in B_{2,
  \infty}^{\boldsymbol{\beta}}(A) : \norm{\alpha}_{B_{2,
    \infty}^{\boldsymbol{\beta}}(A)} \leq L \},
\end{equation*}
where
\begin{equation}\label{eqn:normbesov}
\norm{\alpha}_{B_{2, \infty}^{\boldsymbol{\beta}}(A)} :=  \| \alpha
\|_A + | \alpha
|_{B_{2, \infty}^{\boldsymbol{\beta}}(A)}=\| \alpha
\|_A + \sup_{t>0} \sum_{i=1}^2t^{-\beta_i}\omega_{r_i,i}(g,t).
\end{equation}
Let us denote by $E_{\alpha}$ the integration w.r.t. the joint law
$P_{\alpha}^n$, when the intensity is $\alpha$, of the $n$-sample
$(X_i, N^i(z), Y^i(z); z \leq 1, i=1,\dots,n)$.

\begin{thm}
\label{thm:lower_bound}
There is a positive constant $C_L$ such that
\begin{equation*}
  \inf_{\tilde \alpha} \sup_{\alpha \in B_{2,
      \infty}^{\boldsymbol{\beta}}(A, L)}
  {\mathbb E}_{\alpha} \norm{\tilde \alpha - \alpha}_A^2 \geq C_L n^{-2 \bar
    \beta / (2\bar \beta + 2)}
\end{equation*}
for $n$ large enough, where the infimum is taken among all
estimators and where $C_L$ is a constant that depends on
$\boldsymbol{\beta}, L$ and $A$ only.
\end{thm}

\subsection{Estimation of $f$ and $f_0$}\label{sec:estimation_de_f}

We recall that $f$ is the density of $\mu$, which is defined in
Equation~\eqref{eq:def_f}.  We define
\begin{equation}
\label{repest} \hat f_m=\argmin_{h\in S_m}
\upsilon_n(h) \mbox{ where } \upsilon_n(h)=\|h\|^2-\frac 2n
\sum_{i=1}^n \int_0^1 h(X_i, z)Y^i(z) dz.
\end{equation}
This estimator admits a simple explicit formulation:
\begin{equation}
\label{eq:repest2}
\hat f_m=\sum_{(j,k)\in J_m\times K_m} \hat b_{j,k}\varphi_j^m(x)
\psi_k^m(y), \mbox{ with } \; \hat b_{j,k} = \frac 1n\sum_{i=1}^n
\varphi_j^m(X_i)\int \psi_k^m(z)Y^i(z)dz.
\end{equation}
As before, we consider estimation of $f$ over the compact set $A = [0,
1] \times [0, 1]$. We choose the space $H_{m_2}$ as the space with
maximal dimension, as explained below. Let us denote it by ${\mathcal
  H}_n$, by ${\mathcal D}_n^{(2)} = \dim(\mathcal H_n)$ its dimension
(see~$(\mathcal M1)$) and by $\ell_n$ its index so that $H_{\ell_n} =
\mathcal H_n$. Hence, we consider, instead of a general $\hat f_m$,
the estimator
\begin{equation*}
\hat f_{m_1} := \argmin_{h\in F_{m_1} \times {\mathcal
    H}_n} \upsilon_n(h).
\end{equation*}
We are now in a position to define an estimator of $f_0$ by
considering any $\inf_{(x,z)\in A} \hat f_{m_1}(x,z)$ with a given
$m_1$. Indeed, an arbitrary choice is sufficient for our estimation
problem concerning $f_0$. In our setting, only a rough estimation of
the lower bound on $f$ is useful. Therefore, for the purpose of
estimating $\alpha$, we can define
\begin{eqnarray} \label{condinf} &&\hat f_0 := \inf_{(x,z)\in A} \hat
  f_{m^*_1}(x,z) \mbox{ with } m^*_1=(D_{m_1^*}, \mathcal D_n^{(2)}).
\end{eqnarray}
Then, the following result holds:
\begin{prop} \label{Omegacomp}
Consider  $\hat f_0$ defined by  (\ref{condinf}) in the basis {\rm [T]}, with $\log n\leq
D_{m_1^*} \leq n^{1/4}/ \sqrt{\log n}$ and ${\mathcal D}_n^{(2)}= n^{1/4}/ \sqrt{\log n}$. Assume  that $f\in {\mathcal B}_{2,\infty}^{(\tilde\beta_1,\tilde\beta_2)}(A)$ with  $\bar{\tilde \beta}>1$, then
${\mathbb P}(|\hat f_0-f_0|> f_0/2)\leq C'_k / n^k$, for any
integer $k$, where $C_k$ is a constant and therefore $\hat f_0$ fulfills assumption (${\mathcal A5}$).
\end{prop}
The proof of this result is given in Section \ref{auxiliary}.

\medskip

Hereafter, we develop a remark concerning the estimation of $f$ in
order to explain why we have selected the second dimension $D_{m_2}$
the largest as possible.  Let $f_{m_1}$ be the orthogonal
projection of the restriction of $f$ to $A$ on the space $F_{m_1}
\times \mathcal H_n$, i.e. for $m_n=(m_1, \ell_n)$, $f_{m_1} =
\sum_{(j,k)\in J_{m_1} \times {\mathcal K}_n} b_{j,k}
\varphi_j^{m_n}\psi_k^{m_n}$, with $|J_{m_1}|=D_{m_1}$ and $|{\mathcal
  K}_n|={\mathcal D}_n^{(2)}$. We obtain the following bias-variance
decomposition.
\begin{prop}
\label{riskf}
Under $(\mathcal M1)$, $(\mathcal M2)$, $(\mathcal A1)$ and
$(\mathcal A4)$, we have
\begin{equation}\label{riskdef} {\mathbb E}( \|\hat f_{m_1}-f
  \|^2_A) \leq \|f_{m_1}-f\|^2_A +  \frac{\ell(A_2) \phi_1
    D_{m_1}}n,\end{equation}
where $\ell(A_2)$ is the Lebesgue measure of $A_2$.
\end{prop}

\begin{proof}
We clearly have
\begin{equation}\|\hat f_{m_1}-f\|^2_A =\|f_{m_1}-f\|^2_A + \|\hat f_{m_1}-f_{m_1}\|^2_A,\end{equation}
where the first term is the bias term and
$\|\hat f_{m_1}-f_{m_1}\|^2_A= \sum_{(j,k)\in J_{m_1} \times {\mathcal K}_n} (\hat
b_{j,k}-b_{j,k})^2$ is the variance term. In view
of~\eqref{eq:repest2}, we have $\mathbb E ( \hat b_{j, k} ) = b_{j,
k}$, and, as a consequence:
\begin{eqnarray*}
{\mathbb E}(\|\hat f_{m_1}-f_{m_1}\|^2_A) &=& \sum_{(j,k)\in
J_{m_1} \times {\mathcal K}_n }
{\rm Var}(\hat b_{j,k})
\\ &=&
\sum_{(j,k)\in J_{m_1} \times {\mathcal K}_n } \frac 1n {\rm
Var}\Big(\varphi_j^{m_n}(X_1) \int_{A_2} \psi_k^{m_n}(z)Y^1(z)dz \Big)
\\ &\leq & \sum_{(j,k)\in J_{m_1} \times {\mathcal K}_n } \frac 1n
{\mathbb E}\Big([\varphi_j^{m_n}(X_1)]^2 \Big[\int_{A_2}
\psi_k^{m_n}(z)Y^1(z) dz\Big]^2 \Big)
\end{eqnarray*}
Now, we note that for any $A_2$-square integrable function $\xi$,
$$\sum_{k \in {\mathcal K}_n} \Big[\int_{A_2} \psi_k^{m_n}(z)\xi(z) dz\Big]^2
\leq \int_{A_2} \xi^2(z)dz$$ by a simple projection argument (the left-hand-side term is the squared norm of the projection of $\xi$ on ${\mathcal H}_n$), and thus under assumption $({\mathcal A}4)$,
\begin{eqnarray*}
{\mathbb E}(\|\hat f_{m_1}-f_{m_1}\|^2_A) &\leq & \frac{\ell(A_2)}n
\sum_{j\in J_{m_1}}{\mathbb E}\Big([\varphi_j^{m_1}(X_1)]^2 \Big)
\leq \frac{\ell(A_2)\phi_1 D_{m_1}}n.
\end{eqnarray*}
Gathering the terms, the risk of the estimator is bounded as
in~\eqref{riskdef}.
\end{proof}

Let us discuss the asymptotic rate of estimation of $f_A$, the
restriction of $f$ to $A$, using the above procedure. For that
purpose, assume that $f_A$ belongs to
$B_{2,\infty}^{\boldsymbol{\tilde\beta}}(A)$ with regularity
$\boldsymbol{\tilde\beta}=(\tilde\beta_1,\tilde\beta_2)$.  Now,
consider the collection of trigonometric polynomials for $\varphi_j,
\psi_k$, and apply lemma of \cite{LAC} (see Section~\ref{proofs}
below). The bias term is bounded by
\begin{equation*}
\|f_{m_1}-f\|^2_A \leq C\{D_{m_1}^{-2 \tilde\beta_1} + [{\mathcal
  D}_n^{(2)}]^{-2 \tilde\beta_2}\}.
\end{equation*}
It is worth noticing that the variance term (i.e. the last term of
\eqref{riskdef}) does not depend on $\ell_n$ nor on ${\mathcal
  D}_n^{(2)}$. This explains why the size of the projection space in
the $z$-direction must be chosen the largest as possible, when the
mean square risk is under study. Take ${\mathcal D}_n^{(2)} = \sqrt{n
  / \log n}$ and assume that $\tilde\beta_2 > 1$, then (\ref{riskdef})
becomes
\begin{equation*}
{\mathbb E}( \|\hat f_{m_1}-f\|^2_A) \leq C[ D_{m_1}^{-2 \tilde\beta_1} +
\frac{\ell(A_2) D_{m_1}}n] + \frac{C' \log n} n.
\end{equation*}
Therefore, choosing $D_{m_1^*} = n^{1/(2 \tilde\beta_1+1)}$ gives the rate
$${\mathbb E}( \|\hat f_{m_1}-f_A\|^2) \leq C''n^{-2 \tilde\beta_1/(2 \tilde\beta_1+1)}$$
which is the standard asymptotic rate for a single variable function
with regularity $\tilde\beta_1$. We could study a model selection
procedure and find a penalty function of order $D_{m_1} / n$, so that
a relevant space is chosen in an automatic way. We do not go into
further details since a rough estimation of $f_0$ is sufficient to
estimate the conditional intensity $\alpha$.

\section{Illustration}
\label{simu}

\begin{figure}[htpb]
\includegraphics[scale=0.8]{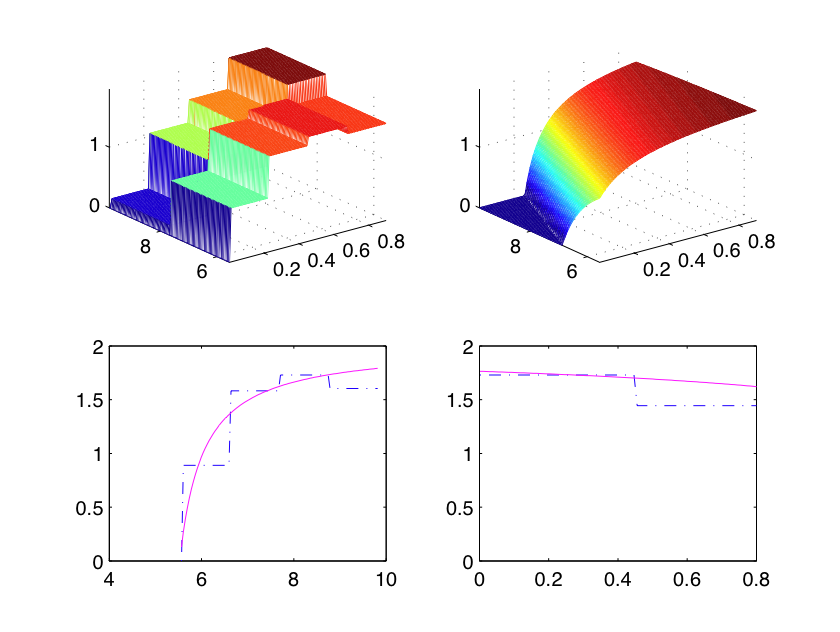}
\caption{Case (NL) Estimated (top left) and true (top right)
  conditional hazard rates and example of sections (bottom) for a
  fixed value of $x$ (left) or $y$ (right).}
\label{fig1}
\end{figure}

In this section, we give a numerical illustration of the adaptive
estimator $\hat \alpha_{\hat m}$, defined in
Section~\ref{procedure}, computed with the dyadic histogram basis
$[H]$. We sample i.i.d. data $(X_1, T_1), \ldots, (X_n, T_n)$ in
three particular cases of the regression model of
Example~\ref{regcens} from Section~\ref{regcens}. For the sake of
simplicity, we simulate the covariates $X_i$ with the uniform
distribution on $[0, 1]$. The size of the data set is $n = 1000$.
\begin{itemize}
\item Case (NL). Non-Linear regression:
\begin{equation*}
  T_i = b(X_i) + \sigma \varepsilon_i.
\end{equation*}
We simulate $\varepsilon_i$ with a $\chi^2(4)$ distribution and
$b(x) = 2x + 5$. Note that in this case, the hazard function to be
estimated is
\begin{equation*}
  \alpha_{\rm NL}(x, t) = \frac{1}{\sigma} \alpha_\varepsilon \Big(
  \frac{t - b(x)}{\sigma} \Big),
\end{equation*}
where $\alpha_\varepsilon$ denotes the hazard function of
$\varepsilon$.
\item Case (AFT). Accelerated Failure Time model:
$$\log(T_i)=a+b X_i + \varepsilon_i,$$ where the $\varepsilon_i$
are standard normal and $a = 5$ and $b = 2$. The hazard function to
be estimated is then:
\begin{equation*}
  \alpha_{AFT}(x,t) = \frac{\alpha_\varepsilon( \log(t) - (a + b
    x))}{t}.
\end{equation*}
\item Case (PH). Proportional Hazards model: in this case, the hazard
  writes
\begin{equation*}
  \alpha(x,t) = \exp(b x) \alpha_0(t).
\end{equation*}
We take $b = 0.4$ and $\alpha_0(t) = a \lambda t^{a-1}$, which is a
Weibull hazard function with $a = 3$ and $\lambda = 1$.
\end{itemize}
The penalty is taken as $$\widehat{{\rm pen}}(m_1,m_2) = 5
\widehat{\|\alpha\|_{\infty,A}} \frac{2^{m_1+m_2}}n,$$ where
$\widehat{\|\alpha\|_{\infty,A}}$ is estimated as the maximal of the
estimated histogram coefficients ($\max_{j,k} \hat a_{j,k}$) on the
largest space which is considered (taken with dimension
$\sqrt{n}$).

We can see from Figures~\ref{fig1}-\ref{fig3} that the algorithm
exploits the opportunity (Figures~\ref{fig1} and~\ref{fig3}) of
choosing different dimensions in the two directions, and that it
captures well the general form of the surfaces.

\begin{figure}[htpb]
\includegraphics[scale=0.8]{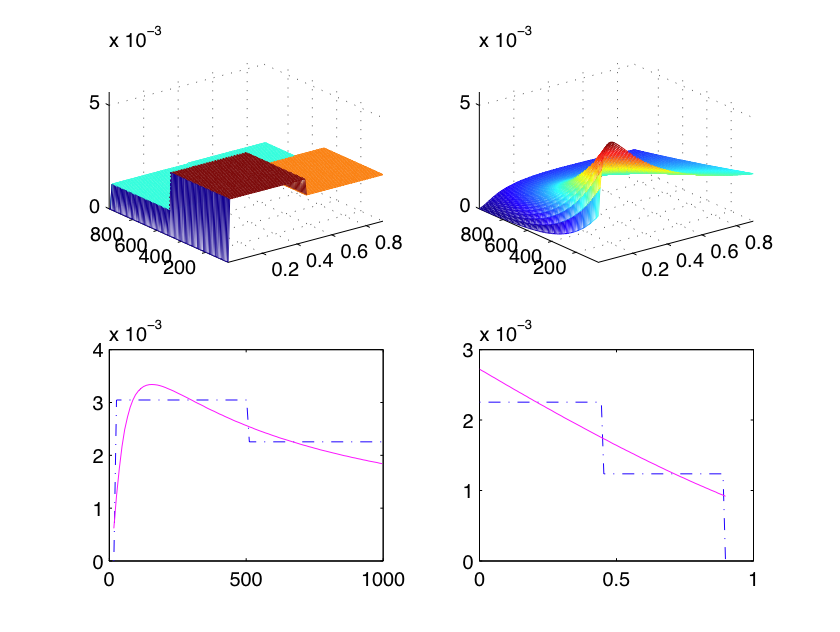}
\caption{Case (AFT) Estimated (top left) and true (top right)
  conditional hazard rates and example of sections (bottom) for a
  fixed value of $x$ (left) or $y$ (right).}
\label{fig2}
\end{figure}

\begin{figure}[htpb]
\includegraphics[scale=0.8]{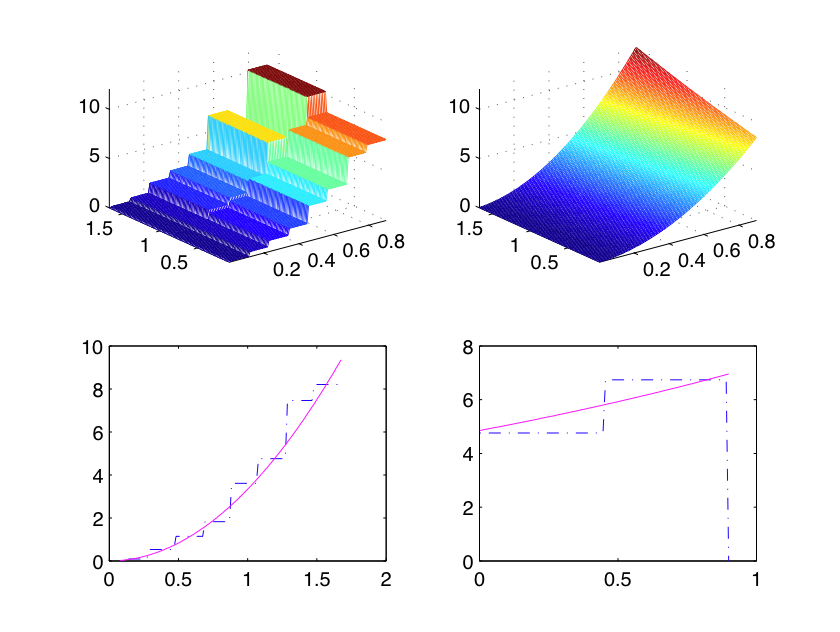}
\caption{Case (PH) Estimated (top left) and true (top right)
  conditional hazard rates and example of sections (bottom) for a
  fixed value of $x$ (left) or $y$ (right).}
\label{fig3}
\end{figure}

\section{Proofs of the main results}
\label{proofs}

\subsection{Proof of Theorem \ref{main}}

We define, for $h_1,h_2$ in $L^2 \cap L^{\infty}(A)$,
the empirical scalar product
\begin{eqnarray}
  \langle h_1,h_2 \rangle_n=\frac 1n \sum_{i=1}^n
  \int_0^1 h_1(X_i,z)h_2(X_i,z)Y^i(z) dz \1( X_i \in [0, 1] )
\end{eqnarray}
and the associated empirical norm $\| h_1 \|_n^2=\langle h_1,h_1 \rangle_n$
which is such that $${\mathbb E}(\|h_1\|_n^2)= \iint_A h_1^2(x,y)d\mu(x,y)
=\iint_A h_1^2(x,y)f(x,y)dxdy=\|h_1\|^2_{\mu}$$ where we recall that $f$
denotes the density of $\mu$ w.r.t. the Lebesgue measure on $A$. We
shall use the following sets:
\begin{eqnarray}
&& \hat\Gamma_m=\{\min{\rm Sp}(G_m)\geq \max(\hat{f}_0
/3,n^{-1/2})\}, \;\; \hat \Gamma := \bigcap_{m\in {\mathcal M}_n}
\hat\Gamma_m,\nonumber \\
&& \label{deltaset} \Delta : =\Big\{ \forall h \in {\mathcal
    S_n} : \Big|\frac{\|h\|_n^2}{\|h\|_{\mu}^2}-1\Big|  \leq \frac 12
\Big\},\text{ and }
\Omega := \Big\{\Big|\frac{\hat f_0}{f_0}-1\Big| \leq \frac
  12\Big\}.
\end{eqnarray}
For $m \in {\mathcal M}_n$, we denote by $\alpha_m$ the orthogonal
projection on $S_m$ of $\alpha$ restricted to $A$. The following
bounds hold:
\begin{eqnarray}\nonumber
{\mathbb E}(\| \hat \alpha_{\hat m} - \alpha \|^2_A) &\leq &
2\|\alpha - \alpha_m \|_A^2 + 2{\mathbb E}(\|\hat\alpha_{\hat
  m}-\alpha_m\|^2_A \1(\Delta \cap \Omega))\\ \nonumber
&  +&2 {\mathbb E}(\|\hat\alpha_{\hat m}
-\alpha_m\|_A^2\1(\Delta^\complement\cap \Omega )) +
2{\mathbb E}(\|\hat\alpha_{\hat m}-\alpha_m\|_A^2\1(\Omega^\complement))
\\ \nonumber &\leq &
2\|\alpha-\alpha_m\|_A^2 + 2{\mathbb E}(\|\hat\alpha_{\hat m}
-\alpha_m\|_A^2\1(\Delta\cap \Omega))\\ \label{etap1}
&  +&4 {\mathbb E}((\|\hat\alpha_{\hat m}\|^2
+\|\alpha\|^2_A)\1(\Delta^\complement\cap \Omega))
+4 {\mathbb E}((\|\hat\alpha_{\hat m}\|^2 +\|\alpha\|^2_A)\1(\Omega^\complement)).
\end{eqnarray}
\noindent We use the following results, whose proofs can be found in
Sections~\ref{proptala} and~\ref{auxiliary}.
\begin{prop}\label{alpha4}
We have ${\mathbb E}(\|\hat \alpha_{\hat m}\|^4)\leq C' n^5,$
where $C'$ is a constant.
\end{prop}

\begin{prop}
\label{deltacomp}
If $(\mathcal M_1)$ is fulfilled, we have ${\mathbb P}(\Delta^\complement)\leq
C_k / n^k$ for any $k \geq 1$, when $n$ is large enough, where $C_k$ is a constant.
\end{prop}
Moreover, $({\mathcal A 5})$ ensures that ${\mathbb P}(\Omega^\complement)\leq C_k/n^k$ for any integer $k$.
Thus, using Propositions~\ref{alpha4} and~\ref{deltacomp} and Assumption $({\mathcal A 5})$, we get
\begin{eqnarray}
\nonumber && {\mathbb E}((\|\hat\alpha_{\hat m}\|^2
+\|\alpha\|^2_A)\1(\Delta^\complement\cap \Omega)) + {\mathbb
  E}((\|\hat\alpha_{\hat m}\|^2 + \|\alpha\|^2_A) \1(\Omega^\complement)) \\
\nonumber & \leq & \|\alpha\|_A^2({\mathbb
  P}(\Omega^\complement)+ {\mathbb P}(\Delta^\complement)) + {\mathbb E}^{1/2}(\|\hat
\alpha_{\hat m}\|^4)({\mathbb P}^{1/2}(\Omega^\complement)+ {\mathbb
  P}^{1/2}(\Delta^\complement)) \\ \label{etap2}&\leq & C_2/n.
\end{eqnarray}
Thus it remains to study ${\mathbb E}(\|\hat\alpha_{\hat m}
-\alpha_m\|_A^2\1(\Delta\cap \Omega))$. We state the following Lemma:
\begin{lem}
\label{inclus}
The following embedding holds:
\begin{equation*}
  \Delta \cap \Omega \subset \hat
  \Gamma \cap \Omega.
\end{equation*}
As a consequence, for all $m \in \mathcal M_n$, the matrices $G_m$ are invertible  on $\Delta\cap \Omega$.
\end{lem}
Let us now define the centered empirical process
\begin{align}
 \nonumber \nu_n(h) &= \frac 1n\sum_{i=1}^n \Big(\int h(X_i,z)
 dN^i(z) - \int
 h(X_i,z) \alpha(X_i,z) Y^i(z) dz \Big) \\
 &= \frac 1n\sum_{i=1}^n \int h(X_i,z) d M^i(z),
 \label{eq:empirical_process}
\end{align}
where we use the Doob-Meyer decomposition. For any $h_1, h_2 \in
(L^2 \cap L^\infty)(A)$, we have
\begin{eqnarray*}
\gamma_n(h_1)-\gamma_n(h_2)&=& \|h_1-h_2\|_n^2
+2\langle h_1-h_2, h_2\rangle_n-\frac 2n\sum_{i=1}^n \int
(h_1-h_2)(X_i,z)dN^i(z) \\ &=&  \|h_1-h_2\|_n^2 +2\langle h_1-h_2,
h_2-\alpha \rangle_n -2\nu_n(h_1-h_2).
\end{eqnarray*}
Now, as on $\Delta\cap\Omega$ we have
$$\gamma_n(\hat\alpha_{\hat m})+{\rm pen}(\hat
m)\leq \gamma_n(\alpha_m)+{\rm pen}(m).$$ It follows, from the inequality $2xy\leq x^2/\theta^2 + \theta^2 y^2$, with $x,y,\theta \in \mathbb R^+$, that, on
$\Delta\cap\Omega$,
\begin{eqnarray*}
\|\hat\alpha_{\hat m}-\alpha_m\|_n^2 &\leq & 2\langle
\hat\alpha_{\hat m}-\alpha_m, \alpha  -\alpha_m\rangle_n + {\rm pen}(m)
+ 2\nu_n(\hat\alpha_{\hat m}-\alpha_m)-{\rm pen}(\hat m)\\ &\leq
& \frac 14 \|\hat\alpha_{\hat m}-\alpha_m\|_n^2
+4 \| \alpha -\alpha_m\|_n^2 + {\rm pen}(m) \\ && + \frac 14
\|\hat\alpha_{\hat m}-\alpha_m\|_{\mu}^2 +
4\sup_{h\in B_{m,\hat m}^{\mu}(0,1)}\nu_n^2(h) -{\rm pen}(\hat m),
\end{eqnarray*}
where $B_{m,m'}^{\mu}(0,1) := \{ h\in S_m+S_{m'} : \|h\|_{\mu} \leq
1\}$. This yields
\begin{eqnarray*}
\frac 34\|\hat\alpha_{\hat m}-\alpha_m\|_n^2 &\leq &
4 \| \alpha -\alpha_m\|_n^2 + {\rm pen}(m) + \frac 14
\|\hat\alpha_{\hat m}-\alpha_m\|_{\mu}^2 \\ && +
4 \Big( \sup_{h\in B_{m, \hat m}^{\mu}(0,1)}\nu_n^2(h) - p(m,\hat
m) \Big) + 4p(m,\hat m) - {\rm pen}(\hat m).
\end{eqnarray*}
Now, let us choose the penalty such that \begin{eqnarray}
\forall m, m', \;
4p(m,m')\leq {\rm pen}(m)+ {\rm pen}(m'),
\end{eqnarray}
and use the definition of
$\Delta$. We obtain on $\Delta \cap \Omega$:
\begin{eqnarray*}
\frac 12\|\hat\alpha_{\hat
  m}-\alpha_m\|_{\mu}^2 & \leq &
4 \| \alpha  - \alpha_m\|_n^2 + 2{\rm pen}(m) \\ && + \frac 14
\|\hat\alpha_{\hat m} - \alpha_m \|_{\mu}^2 + 4 \sum_{m' \in {\mathcal
    M}_n} \Big(\sup_{h\in B_{m, m'}^{\mu}(0,1)}\nu_n^2(h) - p(m, m')
\Big)
\end{eqnarray*}
and thus on $\Delta \cap \Omega$:
\begin{eqnarray*}
\frac 14 \|\hat \alpha_{\hat m} - \alpha_m \|_{\mu}^2  &\leq &
4 \| \alpha - \alpha_m \|_n^2 + 2{\rm pen}(m) \\ && +
4 \sum_{m' \in {\mathcal M}_n} \Big( \sup_{h \in B_{m,
    m'}^{\mu}(0,1)} \nu_n^2(h) - p(m, m') \Big).
\end{eqnarray*}
Using the following proposition, we can achieve the proof of
Theorem~\ref{main}.
\begin{prop}
\label{tala}
Let
\begin{equation*}
  p(m,m') = \kappa (1 + \| \alpha \|_{\infty,A} )\frac{D_m +
    D_{m'}}{n}
\end{equation*}
where $C_0$ is a numerical constant.  Under the assumptions of
Theorem~\ref{main}, we have
\begin{equation*}
  \sum_{m'\in {\mathcal M}_n}{\mathbb E} \Big( \sup_{h\in
    B_{m,m'}^{\mu}(0,1)} (\nu_n^2(h) - p(m,m'))_+ \1(\Delta) \Big)
  \leq \frac{C_1} n.
\end{equation*}

\end{prop}
This proposition entails:
\begin{equation}
\label{fin}
\frac 14 {\mathbb E}(\|\hat\alpha_{\hat
  m}-\alpha_m\|_{\mu}^2 \1(\Delta\cap\Omega)) \leq
4 \| \alpha -\alpha_m\|_{\mu}^2 + 2{\rm pen}(m) + \frac{C_1}n.
\end{equation}
Gathering~\eqref{etap1}, \eqref{etap2} and~\eqref{fin} leads to
\begin{eqnarray}
  \nonumber
  {\mathbb E}(\| \hat \alpha_{\hat m} - \alpha \|^2_A) &\leq &
  2 \| \alpha_m - \alpha \|_A^2 + \frac{8}{f_0}\Big(4 \| \alpha  -
  \alpha_m \|_{\mu}^2 + 2{\rm pen}(m) +  \frac{C_1}n \Big)  +
  \frac{C_2}n  \\ &\leq & 2 \Big(1 + \frac{16 \|f_X\|_{A,\infty}}{f_0} \Big)
  \|\alpha_m - \alpha \|_A^2 +\frac{16}{f_0}{\rm pen}(m) + \frac{C_3}n
\end{eqnarray}
for any $m \in \mathcal M_n$. This concludes the proof of Theorem~\ref{main}. \qed

\subsection{Proof of Corollary \ref{coromain}}\label{coro}

To control the bias term, we state the following lemma proved in \cite{LAC} and following from \cite{HOC} and \cite{NIK}:

\begin{lemn}\cite{LAC} Let  $s$ belong to $B_{2,\infty}^{\boldsymbol{\beta}}(A)$ where $\boldsymbol{\beta}=(\beta_1,\beta_2)$.
We consider that $S_m'$ is one of the following spaces on $A$ of dimension $D_{m_1}D_{m_2}$
:\begin{itemize} \item a space of piecewise polynomials of degrees
bounded by $s_i>\beta_i -1$ ($i=1,2$) based on a partition with
rectangles of sidelengthes $1/D_{m_1}$ and $1/D_{m_2}$, \item a
linear span of $\{\phi_\lambda\psi_\mu, \lambda\in
\cup_0^{m_1}\Lambda(j), \mu\in \cup_0^{m_2}M(k)\}$ where
$\{\phi_\lambda\}$ and $\{\psi_\mu\}$ are orthonormal wavelet
bases of respective regularities $s_1>\beta_1 -1$ and
$s_2>\beta_2 -1$ (here $D_{m_i}=2^{m_i}, i=1,2$), \item the space
of trigonometric polynomials with degree smaller than $D_{m_1}$ in
the first direction and smaller than $D_{m_2}$ in the second
direction.
\end{itemize}
Let $s_m$ be the orthogonal projection of $s$ on $S_m'$.
Then, there exists a positive constant $C_0$ such that
\begin{eqnarray*}
\|s-s_m\|_A=\Big(\int_A|s-s_m|^2\Big)^{1/2}\leq C_0 [ D_{m_1}^{-\beta_1} +D_{m_2}^{-\beta_2} ].
\end{eqnarray*}
\end{lemn}

If we choose for $S_m$ as one of the $S_m'$s, we can apply the above lemma to the function $\alpha_A$,
the restriction of $\alpha$ to $A$. As $\alpha_m$ has been defined as the orthogonal projection of $\alpha_A$ on $S_m$, we get:
\begin{eqnarray*}\|\alpha-\alpha_m\|_A\leq C_0 [ D_{m_1}^{-\beta_1} +D_{m_2}^{-\beta_2} ].\end{eqnarray*}
Now, according to  Theorem \ref{main}, we obtain:
\begin{eqnarray*}\mathbb{E}\|\hat\alpha_{\hat m}-\alpha\|_A^2 \leq C''\underset{m\in\mathcal{M}_n}{\inf}\Big\{D_{m_1}^{-2\beta_1}
  +D_{m_2}^{-2\beta_2}+\frac{D_{m_1}D_{m_2}}{n}\Big\}.\end{eqnarray*}
In particular, if $m^*=(m_1^*,m_2^*)$ is such that
\begin{eqnarray*}
D_{m_1^*}=\lfloor n^{\frac{\beta_2}{\beta_1+\beta_2+2\beta_1\beta_2}}\rfloor \text{ and }
D_{m_2^*}=\lfloor (D_{m_1^*})^{\frac{\beta_1}{\beta_2}}\rfloor
\end{eqnarray*}
then
\begin{eqnarray*}
\mathbb{E}\|\hat\alpha_{\hat m}-\alpha\|_A^2 \leq
C'''\Big\{D_{m_1^*}^{-2\beta_1} +\frac{D_{m_1^*}^{1+\beta_1/\beta_2}}{n}\Big\}
=O\Big(n^{-\frac{2\beta_1\beta_2}{\beta_1+\beta_2+2\beta_1\beta_2}}\Big)=O(n^{-\frac{2\bar\beta}{2\bar\beta+2}}),\end{eqnarray*}
where the harmonic mean of $\beta_1$ and $\beta_2$ is
$\bar\beta={2\beta_1\beta_2}/({\beta_1+\beta_2}).$
The condition $D_{m_1}\leq n^{1/2}/\log n$  allows this choice of $m$ only if
$\beta_2/(\beta_1+\beta_2+2\beta_1\beta_2)<1/2$ i.e. if $\beta_1-\beta_2+2\beta_1\beta_2>0$.
In the same manner, the condition $\beta_2-\beta_1+2\beta_1\beta_2>0$ must be verified.
Both conditions hold if $\beta_1>1/2$ and $\beta_2>1/2$.

\subsection{Proof of Theorem \ref{thm:lower_bound}}

In order prove Theorem~\ref{thm:lower_bound}, we use the following
theorem from~\cite{tsybakov03}, which is a standard tool for the proof
of such a lower bound. We say that $\partial$ is a
\emph{semi-distance} on some set $\Theta$ if it is symmetric and if it
satisfies the triangle inequality and $\partial(\theta, \theta) = 0$
for any $\theta \in \Theta$. We consider $K(P, Q) := \int \log (dP/dQ)
dP$ the Kullback-Leibler divergence between probability measures $P$
and $Q$ such that $P \ll Q$.

\begin{thmn}[\cite{tsybakov03}]
Let $(\Theta, \partial)$ be a set endowed with a semi-distance
$\partial$. We suppose that $\{ P_\theta : \theta \in \Theta \}$ is
a family of probability measures on a measurable space $(\mathcal X,
\mathcal A)$ and that $v > 0$. If there exist
$\{ \theta_0, \ldots, \theta_M \} \subset \Theta$, with $M \geq 2$,
such that
\begin{enumerate}
\item $\partial(\theta_j, \theta_k) \geq 2 v \quad \forall\; 0 \leq j
  < k \leq M$
\item $P_{\theta_j} \ll P_{\theta_0} \quad \forall\; 1 \leq j \leq M$,
\item $\frac{1}{M} \sum_{j=1}^M K(P_{\theta_j}, P_{\theta_0})\leq a
  \log (M)$ for some $a \in (0, 1/8)$,
\end{enumerate}
then
\begin{equation*}
  \inf_{\hat \theta }\sup_{\theta \in \Theta} E_\theta
  [  ( v^{-1} \partial( \hat \theta, \theta) )^2 ] \geq
  \frac{\sqrt{M}}{1 + \sqrt{M}} \bigg( 1 - 2 a - 2
  \sqrt{ \frac{a}{\log (M)}} \bigg),
\end{equation*}
where the infimum is taken among all estimators.
\end{thmn}

We construct a family of functions $\{ \alpha_0, \ldots, \alpha_M \}$
that satisfies points~(1)--(3). Let $\alpha_0(x, t) = |B|^{-1} \1(t
\in B)$ where $B$ is a compact set such that $A=A_1 \times A_2 \subset
B \times B$ and $|B| \geq 2 |A|^{1/2} / L$. As a consequence, we have
$\alpha_0(x, t) > 0$ for $(x, t) \in A$ and $\norm{\alpha_0}_{B_{2,
   \infty}^{\boldsymbol{\beta}}(A)} = \norm{\alpha_0}_A +
|\alpha_0|_{B_{2, \infty}^{\boldsymbol{\beta}}(A)} \leq L / 2$ since
$|\alpha_0|_{B_{2, \infty}^{\boldsymbol{\beta}}(A)} = 0$, see
\eqref{eqn:normbesov}. We shall denote for short $a_0 = |B|^{-1}$ in
the following. Let $\psi$ be a very regular wavelet with compact
support (the Daubechies's wavelet for instance), and for $j = (j_1,
j_2) \in \mathbb Z^2$ and $k = (k_1, k_2) \in \mathbb Z^2$, let us
consider
\begin{equation*}
 \psi_{j, k}(x, t) = 2^{(j_1 + j_2)/2} \psi( 2^{j_1} t - k_1)
 \psi(2^{j_2} x - k_2).
\end{equation*}
Let $S_{j, k}$ stands for the support of $\psi_{j,k}$. We consider the
maximal set $K_j \subset \mathbb Z^2$ such that
\begin{equation}
\label{eq:R_j_def}
S_{j, k} \subset A, \forall k \in R_j \text{ and } S_{j,k}
\cap S_{j, k'} = \emptyset, \forall k, k' \in R_j, k \neq k'.
\end{equation}
The cardinality of $R_j$ satisfies $|R_j| = c 2^{j_1 + j_2}$, where
$c$ is a positive constant that depends on $A$ and on the support of
$\psi$ only. Consider the set $\Omega_j = \{ 0, 1 \}^{|R_j|}$ and
define for any $\omega = (\omega_k) \in \Omega_j$
\begin{equation*}
\alpha(\cdot; \omega) := \alpha_0 + \sqrt{\frac{b}{n}} \sum_{k \in
  R_j} \omega_k \psi_{j, k},
\end{equation*}
where $b > 0$ is some constant to be chosen below. In view
of~\eqref{eq:R_j_def} we have
\begin{equation*}
\norm{ \alpha(\cdot; \omega) - \alpha(\cdot; \omega') }_A^2 =
\frac{b \rho(\omega, \omega')}{n}
\end{equation*}
where
\begin{equation*}
\rho(\omega, \omega') := \sum_{k \in R_j} \1( \omega_k \neq
\omega'_{k})
\end{equation*}
is the Hamming distance on $\Omega_j$. Using a result of
Varshamov-Gilbert - see~\cite{tsybakov03} - we can find a subset $\{
\omega^{(0)}, \ldots, \omega^{(M_j)} \}$ of $\Omega_j$ such that
\begin{equation*}
\omega^{(0)} = ( 0, \ldots, 0), \quad \rho(\omega^{(p)}, \omega^{(q)})
\geq |R_j| / 8
\end{equation*}
for any $0 \leq p < q \leq M_j$, where $M_j \geq 2^{|R_j|/8}$. We
consider the family $\mathcal A_j = \{ \alpha_0, \ldots, \alpha_{M_j}
\}$ where $\alpha_p = \alpha(\cdot, \omega^{(p)})$. This family
satisfies for any $0 \leq p < q \leq M_j$
\begin{equation*}
\norm{ \alpha_p - \alpha_q }_A \geq \Big( \frac{b |R_j|}{8 n}
\Big)^{1/2} = 2 v_j
\end{equation*}
for $v_j := \sqrt{b |R_j| / (32 n)}$. This proves point~(1). Now, let
us gather here some properties for this family of functions. We have
\begin{equation*}
 \norm{\alpha(\cdot; \omega) - \alpha_0}_{\infty, A} \leq \sqrt{
   \frac{b 2^{(j_1 + j_2)}}{n} } \norm{\psi}_\infty^2 \leq a_0 / 3
\end{equation*}
and consequently $\alpha(x, t; \omega) \geq 2 a_0 / 3 > 0$ for any
$(x, t) \in A$ and $\omega \in \Omega_j$ whenever
\begin{equation}
\label{eq:lower_bound_condj1}
\Big( \frac{b 2^{j_1+j_2}}{n} \Big)^{1/2} \leq \frac{a_0}{3
  \norm{\psi}_\infty^2}.
\end{equation}
Using~\cite{HOC}, we have for $\psi$ smooth enough that
\begin{align*}
 \norm{\sum_{k \in R_j} \omega_k \psi_{j, k}}_{B_{2,
     \infty}^{\boldsymbol{\beta}}(A)} \leq (2^{j_1 \beta_1} + 2^{j_2
   \beta_2}) \norm{\sum_{k \in R_j} \omega_k \psi_{j, k}}_A \leq
 (2^{j_1 \beta_1} + 2^{j_2 \beta_2}) ( c 2^{j_1 + j_2} )^{1/2}.
\end{align*}
Hence, if
\begin{equation}
\label{eq:lower_bound_condj2}
\frac{(2^{j_1 \beta_1} + 2^{j_2  \beta_2}) ( 2^{j_1 +
    j_2} )^{1/2}}{\sqrt{n}} \leq \frac{L}{2 \sqrt{b c}},
\end{equation}
we have $\norm{\alpha(\cdot; \omega)}_{B_{2,
   \infty}^{\boldsymbol{\beta}}(A)} \leq L$, so $\alpha(\cdot;
\omega) \in B_{2, \infty}^{\boldsymbol{\beta}}(A,L)$ for any $\omega
\in \Omega_j$. This proves that $\mathcal A_j \subset B_{2,
 \infty}^{\boldsymbol{\beta}}(A,L)$.

Points~(2) and (3) are derived using Jacod's formula
(see~\cite{ABGK}). Indeed, we can prove that the log-likelihood
$\ell(\alpha, \alpha_0) := \log (d P_{\alpha} / d P_{\alpha_0})$ of
$N$ writes
\begin{equation*}
 \ell(\alpha, \alpha_0) = \int_0^1 (\log \alpha(X, t) - \log
 \alpha_0(X, t)) d N(t) - \int_0^1 (\alpha(X, t) - \alpha_0(X, t))
 Y(t) dt.
\end{equation*}
For any $\alpha \in \mathcal A_j$, we have $\norm{\alpha -
 \alpha_0}_{\infty, A} \leq a_0 / 3 \leq \alpha(x, t) / 2$ for any
$(x, t) \in A$. The Doob-Meyer decomposition allows to write that,
under $P_{\alpha_0}$:
\begin{align*}
 \ell(\alpha, \alpha_0) &= \int_0^1 \Big( \Phi_{1/\alpha(X,
   t)}(\alpha(X, t) - \alpha_0(X, t)) - (\alpha(X, t) -
 \alpha_0(X, t)) \Big) Y(t) dt \\
 &+ \int_0^1 (\log \alpha(X, t) - \log \alpha_0(X, t)) dM(t)
\end{align*}
where $\Phi_a(x) := - \log(1 - ax) / a$ for $a > 0$ and $x < 1 /
a$. But since $\Phi_a(x) \leq x + a x^2$ for any $x \leq 1/(2a)$, we
obtain
\begin{equation*}
 \ell(\alpha, \alpha_0) \leq \frac{3}{2 a_0} \int_0^1
 (\alpha(t, X) - \alpha_0(t, X))^2 Y(t) dt + \int_0^1 (\log
 \alpha_0(t, X) - \log \alpha(t,X)) dM(t)
\end{equation*}
which gives by integration with respect to $P_{\alpha}$
\begin{equation*}
 K(P_{\alpha}, P_{\alpha_0}) \leq \frac{3 \norm{\alpha -
     \alpha_0}_\mu^2}{2 a_0} \leq \frac{3 \norm{f_X}_\infty \norm{\alpha -
     \alpha_0}_A^2}{2 a_0} \leq \frac{3 b \norm{f_X}_\infty |R_j|}{2 n a_0},
\end{equation*}
for any $\alpha \in \mathcal A_j$. Since the counting processes $(N^1,
\ldots, N^n)$ are independent, we have $K(P_{\alpha}^n,
P_{\alpha_0}^n) = n K(P_{\alpha}, P_{\alpha_0})$ and
\begin{equation*}
 \frac{1}{M} \sum_{p=0}^M K(P_{\alpha_p}^n, P_{\alpha_0}^n) \leq
 \frac{3 b \norm{f_X}_\infty |R_j| }{2a_0} \leq a \log M_j
\end{equation*}
with $a = 12 b \norm{f_X}_\infty / (a_0 \log 2) \in (0, 1/8)$ for $b$
small enough. It only remains to choose the levels $j_1$ and $j_2$ so
that~\eqref{eq:lower_bound_condj1} and \eqref{eq:lower_bound_condj2}
holds, and to compute the corresponding $v_j$. We take $j = (j_1,
j_2)$ such that
\begin{equation*}
 c_1 / 2 \leq 2^{j_1} n^{-\beta_2 / (\beta_1 + \beta_2 + 2
   \beta_1 \beta_2)} \leq c_1 \text{ and }   c_2 / 2 \leq 2^{j_2}
 n^{-\beta_1 / (\beta_1 + \beta_2 + 2 \beta_1 \beta_2)} \leq c_2
\end{equation*}
where $c_1$ and $c_2$ are positive constants satisfying
$(c_1^{\beta_1} + c_2^{\beta_2}) \sqrt{c_1 c_2} \leq L / (2 \sqrt{b
 c})^{1/2}$. For this choice, $2^{j_1 + j_2} / n \leq c_1 c_2 n^{-2
 \bar \beta / (2 \bar \beta + 2)}$ so~\eqref{eq:lower_bound_condj1}
holds for $n$ large enough and \eqref{eq:lower_bound_condj2} holds and
$v_j \geq c_3 n^{-\bar \beta / (2 \bar \beta + 2)}$ where $c_3 =
\sqrt{b c c_1 c_2 / 128}$. $\hfill \square$

\section{Deviation and maximal inequalities for the empirical process}
\label{sec:deviation}

Usually, in model selection (see for instance~\cite{massart}), the
penalty is explained using the so-called Talagrand's deviation
inequality for the maximum of empirical processes. Because the
empirical process $\nu(\cdot)$ (see
Equation~\eqref{eq:empirical_process}) considered here has a
particular structure, we cannot use directly Talagrand's inequality.
In this Section, we prove Bennett and Bernstein inequalities for
$\nu_n(\cdot)$, and derive a maximal bound using the so-called
chaining technique which explains the penalty~\eqref{penalite}.

\subsection{Deviation inequality}

\begin{lem}
\label{martingale}
For any positive $\delta$, $\epsilon$ and for any function $h \in
(L^2 \cap L^\infty)(A)$, we have the following Bennett-type
deviation inequality\textup:
\begin{equation*}
  {\mathbb P} \big( \nu_n(h) \geq \epsilon , \|h\|_n \leq \delta
  \big)\leq \exp\Big( -\frac{n \delta^2
    \norm{\alpha}_{\infty,A}}{\norm{h}_{\infty,A}^2} g \Big( \frac{
    \epsilon \norm{h}_{\infty,A}}{\norm{\alpha}_{\infty,A} \delta^2 }
  \Big) \Big)
\end{equation*}
where $g(x) = (1+x) \log(1+x) - x$ for any $x \geq 0$. As a
consequence, we obtain the following Bernstein-type inequalities:
\begin{equation}
  \label{bernineg}
  {\mathbb P} \big( \nu_n(h) \geq \epsilon , \|h\|_n \leq \delta
  \big)\leq \exp \Big(
  -\frac{n\epsilon^2/2}{\norm{\alpha}_{A,\infty} \delta^2  + \frac
    13 \epsilon \norm{h}_{A,\infty}} \Big),
\end{equation}
and
\begin{equation}
  \label{bernineg2}
  {\mathbb P} \Big( \nu_n(h)\geq \delta \sqrt{\|\alpha\|_{\infty,A}
    x}+ \|h\|_{\infty,A}x/3,\;\;
  \|h\|_n^2\leq \delta^2 \Big)\leq \exp(-nx).
\end{equation}
\end{lem}

\begin{proof}
Remark that $\nu_n(h) = \nu(h, 1)$ where $\nu(h, \cdot)$ is the
stochastic process given by
\begin{equation*}
 n \nu(h, t) :=  \sum_{i=1}^n \int_0^t
 h(X_i,z) d M^i(z):=n \sum_{i=1}^n  \nu(h, t)^i.
\end{equation*}
The predictable variation of $M^i$ is given by $\langle M^i(t)
\rangle = \int_0^t \alpha(X_i,z) Y^i(z) dz$, so we have
\begin{equation*}
  \langle n \nu(h, t)^i \rangle =
  \int_0^t h(X_i,z)^2 \alpha(X_i,z) Y^i(z) dz
\end{equation*}
for any $t \in [0, 1]$. Moreover, we have $\Delta M^i(t) \in \{ 0, 1
\}$ for any $i=1, \dots,n$ since the counting processes $N^i$ admit
intensities. We can write $ \nu(h, t)^i = \nu(h, t)^{i,c} +
\nu(h,t)^{i,d}$ where $\nu(h, t)^{i,c}$ is a continuous martingale
and where $ \nu(h,t)^{i,d}$ is a purely discrete martingale (see
e.g. \cite{lipstershiryayev}). For some $a > 0$ (to be chosen later
on) we define $U_a^i(t) := a n \nu^i(h, t) - S_a^i(t)$, where
$S_a^i(t)$ is the compensator of
\begin{equation}
  \label{eq:S_h_def}
  \frac12  \langle a n \nu(h, t)^{i,c} \rangle+ \sum_{s \leq t}
  \Big(\exp(a | \Delta n \nu(h, s)^{i}|) -1 -a| \Delta n \nu(h, s)^{i}| \Big).
\end{equation}
We know from the proof of Lemma 2.2 and Corollary 2.3 of
\cite{vandegeer95}, that $\exp(U_a^i(t))$ is a
supermartingale. Using the standard Cram\'er-Chernoff method (see
for instance \cite{massart}, Chapter~2), we have, for any $a > 0$:
\begin{align*}
& \mathbb P \Big(\nu_n(h) \geq \epsilon , ||h||_n \leq \delta \Big) \\
  &=\mathbb P \Big( \exp(
  a n\nu_n(h)) \geq \exp(na \epsilon ) , ||h||_n \leq \delta \Big) \\
  &\leq \Big( \mathbb E \Big[ \exp \Big(an \sum_{i=1}^n  \nu(h, 1)^i
  - \sum_{i=1}^n S_a^i(1)\Big) \Big] \Big)^{1/2} \Big( \mathbb E \Big[
  \exp\Big(\sum_{i=1}^n S_a^i(1)- a n \epsilon \Big) \1\{||h||_n
    \leq \delta\} \Big] \Big)^{1/2} \\
    &\leq \Big( \mathbb E \Big[ \exp\Big(\sum_{i=1}^n S_a^i(1) - a n \epsilon
  \Big) \1\{||h||_n \leq \delta\} \Big]\Big)^{1/2}.
\end{align*}
The last inequality holds since $\exp(U_a^i(t))=\exp( a n \nu^i(h,
t) - S_a^i(t))$ are independent supermartingales with $U_a^i(0)=0$,
so that $\mathbb E [ \exp(U_a^i(t))] \leq 1$, for $i=1, \dots,n$.

Let us decompose $M^i= M^{i,c}+ M^{i,d}$, with $M^{i,c}$ a
continuous martingale and $M^{i,d}$ a purely discrete martingale.
The process $V_2^i(t) := \langle M^i(t) \rangle$ is the compensator
of the quadratic variation process $[M^i(t)]= \langle  M^{i,c}(t)
\rangle+ \sum_{s \leq t} | \Delta M^i(t)|^2$. If $k \geq 3$, we
define $V^i_k(t)$ as the compensator of the $k$-variation process
$\sum_{s \leq t} | \Delta M^i(t)|^k$ of $ M^i(t)$. Since $\Delta
M^i(t) \in \{ 0, 1 \}$ for all $0 \leq t \leq 1$, the $V_k^i$ are
all equal for $k \geq 3$ and such that $V^i_k(t) \leq V_2^i(t)$, for
all $k \geq 3$. The process $S_a^i(1)$ has been defined as the
compensator of~\eqref{eq:S_h_def}. As a consequence, we have:
 \begin{align*}
   S_a^i(1)
   = 
   \sum_{k \geq 2}\frac{a^k}{k!}
   \int_0^1 |h(X_i,z)|^k dV^i_k(z)
   \leq \int_0^1 h(X_i,z)^2 dV^i_2(z) \times \sum_{k \geq
     2}\frac{\norm{h}_{\infty,A}^{k-2}}{k!}a^k
 \end{align*}
 and if $\norm{h}_n \leq \delta$
 \begin{equation*}
   \sum_{i=1}^n S_a^i(1) \leq \bar S_a^n := \frac{n
     \delta^2 \norm{\alpha}_{\infty,A}}{\norm{h}_{\infty,A}^2} \Big(
   \exp\big( a \norm{h}_{\infty,A} \big) - 1 - a
   \norm{h}_{\infty,A} \Big).
 \end{equation*}
 The minimum of $\bar S_a^n - an \epsilon$ for $a > 0$ is achieved by
 \begin{equation*}
   a = \frac{1}{\norm{h}_{\infty,A}} \log
   \Big(\frac{ \epsilon \norm{h}_{\infty,A} }{\norm{\alpha}_{\infty,A}
     \delta^2 } +1 \Big)
 \end{equation*}
 and is equal to
 \begin{equation*}
   -\frac{n \delta^2 \norm{\alpha}_{\infty,A}}{\norm{h}_{\infty,A}^2}
   g\Big(\frac{ \epsilon \norm{h}_{\infty,A}
   }{\norm{\alpha}_{\infty,A} \delta^2 } \Big)
 \end{equation*}
 where we recall that $g(x) = (1+x) \log(1+x) - x$. This concludes
 the proof of the Bennett inequality. Inequality~(\ref{bernineg})
 follows from the fact that $g(x) \geq 3 x^2 / (2(x + 3))$ for any $x
 \geq 0$. To prove~\eqref{bernineg2}, we use the following trick
 from~\cite{BM8}: we have $g(x) \geq g_2(x)$ for any $x \geq 0$ where
 $g_2(x) := x + 1 - \sqrt{1 + 2x}$ and $g_2^{-1}(y) = \sqrt{2y} + y$.
\end{proof}

\subsection{Proof of Proposition~\ref{tala} (maximal inequality via
 $L^2 - L^\infty$ chaining)}
\label{proptala}


Using a $L^2 - L^\infty$ chaining method, as in~\cite{BBM} or
\cite{COM}, we obtain the following result, which leads to
Proposition~\eqref{tala}:
\begin{lem}
\label{propsup}
Let $B_{m,m'}^{\mu}(0,1) = \{t\in S_m+S_{m'}, \|t\|_{\mu} \leq 1\}$. Then
\begin{equation*}
  {\mathbb E} \Big( \sup_{h\in B_{m,m'}(0,1)} (\nu_n^2(h)-
  p(m,m'))_+ \1(\Delta) \Big) \leq C(1 + \|\alpha\|_{\infty, A})
  \frac{e^{-D_{m'}}}n,
\end{equation*}
where
\begin{equation*}
  p(m,m') = \kappa (1 + \| \alpha \|_{\infty,A} )\frac{D_m +
    D_{m'}}{n}.
\end{equation*}
\end{lem}

\begin{proof}
 The result of Lemma~\ref{propsup} is obtained from
 Inequality~\eqref{bernineg} by a $L^2(\mu)-L^{\infty}$ chaining
 technique. The method is analogous to the one given in Proposition 4 p.~282-287 in Comte~(2001), in Theorem~5 in Birg\'e and
 Massart~(1998) and in Proposition~7, Theorem~8 and Theorem~9 in
 Barron et al.~(1999). Since the context is different, we give, for
 the sake of completeness, the details of the proof. It relies on the
 following lemma (Lemma 9 in Barron {\it et al.}~(1999)):

\begin{lemn}[\cite{BBM}]
 Let $\mu$ be a positive measure on $[0,1]$. Let
 $(\psi_{\lambda})_{\lambda\in \Lambda}$ be a finite orthonormal
 system in $L^2\cap L^{\infty}(\mu)$ with $|\Lambda|=D$ and $\bar S$
 be the linear span of $\{\psi_{\lambda}\}$. Let
 \begin{equation}
   \label{rbar}
   \bar r = \frac 1{\sqrt{D}} \sup_{\beta\neq 0}
   \frac{\|\sum_{\lambda\in
       \Lambda}\beta_{\lambda}\psi_{\lambda}\|_{\infty}}{|\beta|_{\infty}}.
 \end{equation}
 For any positive $\delta$, one can find a countable set $T\subset
 \bar S$ and a mapping $p$ from $\bar S$ to $T$ with the following
 properties:
 \begin{itemize}
 \item for any ball ${\mathcal B}$ with radius $\sigma\geq 5\delta$,
   \begin{eqnarray*}
     |T\cap {\mathcal B}|\leq (B'\sigma/\delta)^D \;
     \mbox{ with } \; B'<5,
   \end{eqnarray*}
 \item $\|u-p(u)\|_{\mu}\leq \delta$ for all $u$ in $\bar S$, and
   \begin{eqnarray*}
     \sup_{u\in p^{-1}(t)} \|u-t\|_{\infty} \leq \bar
     r \delta, \; \mbox{ for all } t \mbox{ in } T.
   \end{eqnarray*}
 \end{itemize}
\end{lemn}

To use this lemma, the main difficulty is often to evaluate $\bar r$
in the different contexts.  We consider a collection of product
models $(S_m)_{m\in {\mathcal M}_n}$ which can be [DP] or [T]. For the sake of place, we omit collection [W] as it right similar to collection [DP].
Recall that $B_{m,m'}^{\mu}(0,1)=\{t\in S_m+ S_{m'}, \|t\|_{\mu}\leq
1\}$. We have to compute $\bar r=\bar r_{m,m'}$ corresponding to
$\bar S=S_m+S_{m'}\subset {\mathcal S}_n$ on which the norm
connection holds.  We denote by $D(m,m') = \dim(S_m+S_{m'})$.
\begin{itemize}
\item Collection [DP] -- As $S_m+S_{m'}$ is a linear space, an
 orthonormal $L^2(\mu)$-basis $(\psi_{\lambda})_{\lambda\in
   \Lambda_n}$ can be built by orthonormalisation on each
 sub-rectangle of $(\varphi_{\lambda})_{\lambda\in \Lambda_n}$, the
 orthonormal basis of ${\mathcal S}_n$. Then
\begin{eqnarray*}
  \sup_{\beta\neq 0} \frac{\|\sum_{\lambda\in
      \Lambda_n}\beta_{\lambda}\psi_{\lambda}
    \|_{\infty,A}}{|\beta|_{\infty}} &\leq & \| \sum_{\lambda\in
    \Lambda_n}|\psi_{\lambda}| \|_{\infty, A} \leq
  (r+1) \sup_{\lambda\in \Lambda_n}\|\psi_{\lambda} \|_{\infty,A} \\
  &\leq & (r+1)^{3/2}\sqrt{N_n}\sup_{\lambda\in \Lambda_n}
  \|\psi_{\lambda}\| \\ &\leq &
  (r+1)^{3/2}\sqrt{N_n}\sup_{\lambda\in \Lambda_n}
  \|\psi_{\lambda}\|_{\mu}/\sqrt{f_0}\\ &\leq &
  (r+1)^{3/2}\sqrt{N_n/f_0}.\end{eqnarray*} Thus here $\bar
r_{m,m'}\leq ((r+1)^{3/2}/\sqrt{f_0})\sqrt{N_n/D(m,m')}$.
\item Collection [T]-- For trigonometric polynomials, we write
\begin{eqnarray*} \sup_{\beta\neq 0}
\frac{\|\sum_{\lambda\in \Lambda_n}\beta_{\lambda}\psi_{\lambda}\|_{\infty,A}}{|\beta|_{\infty}}
& \leq & \frac{C\sqrt{N_n}\|\sum_{\lambda}\beta_{\lambda}\psi_{\lambda}\|}{|\beta|_{\infty}} \leq
\frac{C\sqrt{N_n}\|\sum_{\lambda}\beta_{\lambda}\psi_{\lambda}\|_{\mu}}{\sqrt{f_0}|\beta|_{\infty}}
\\ &\leq &
\frac{C\sqrt{N_n}\sqrt{\sum_{\lambda}\beta_{\lambda}^2}}{\sqrt{f_0}|\beta|_{\infty}}\leq \frac{C\sqrt{N_n D(m,m')}}{\sqrt{f_0}}.
\end{eqnarray*}
Therefore, $\bar r_{m,m'}\leq C\sqrt{N_n/f_0}$.\end{itemize}

We may now prove Lemma~\ref{propsup}.  We apply the Lemma
from~\cite{BBM} to the linear space $S_m+ S_{m'}$ of dimension
$D(m,m')$ and norm connection measured by $\bar r_{m,m'}$ bounded
above. We consider $\delta_k$-nets $T_k= T_{\delta_{k}}\cap
B_{m,m'}^{\mu }(0,1)$, with $\delta_k = \delta_0 2^{-k}$ and $\delta_0
\leq 1/5$ (to be chosen later). Moreover we set $H_k=\log(|T_k|)\leq
D(m,m')\log(5/\delta_k) = D(m,m')[k\log(2)+\log(5/\delta_0)]$.  Given
some point $h\in B_{m,m'}^{\mu}(0,1)$, we can find a sequence
$\{h_k\}_{k\geq 0}$ with $h_k\in T_k$ such that $\|h-h_k\|^2_{\mu
}\leq \delta_k^2$ and $\|h-h_k\|_{\infty,A}\leq \bar r_{m,m'}
\delta_k$.  Thus we have the following decomposition that holds for
any $h \in B_{m,m'}^{\mu}(0,1)$:
\begin{equation*}
 h = h_0 + \sum_{k \geq 1} (h_{k} - h_{k-1}),
\end{equation*}
with $\| h_0 \|_{\mu } \leq 1$, $\| h_0 \|_{\infty, A} \leq \bar
r_{(m,m')}$, and
\begin{equation*}
 \| h_{k}-h_{k-1}\|^2_{\mu} \leq 2(\delta_k^2 +
 \delta_{k-1}^2 ) = 5\delta_{k-1}^2/2, \quad \|h_k - h_{k-1}
 \|_{\infty,A} \leq 3 \bar r_{(m,m')}\delta_{k-1}/2
\end{equation*}
for any $k \geq 1$. In the sequel we denote by ${\mathbb
 P}_\Delta(\cdot)$ the measure ${\mathbb P}(\cdot \cap \Delta)$,
see~\eqref{deltaset}.
Let in addition $(\eta_k)_{k\geq
 0}$ be a sequence of positive numbers that will be chosen later on
and $\eta$ such that $\eta_0+\sum_{k\geq 1}\eta_k \leq \eta$.  We
have:
\begin{eqnarray*}
&& \mathbb P_\Delta\Big[\sup_{h\in  B_{m,m'}^{\mu}(0,1)}\nu_n(h)>
\eta \Big]\\
&=& \mathbb P_\Delta\Big[\exists (h_k)_{k\in \text{I\negthinspace N}}\in
\prod_{k\in \text{I\negthinspace N}}T_{k}\ /\ \nu_n(h_0)
+\sum_{k=1}^{+\infty}\nu_n(h_k-h_{k-1})> \eta_0 + \sum_{k\geq 1}\eta_k \Big]\\
&\leq &  \mathbb P_1 + \mathbb P_2
\end{eqnarray*}
where
\begin{eqnarray*}
 \mathbb P_1 = \sum_{h_0\in T_0} \mathbb P_\Delta(\nu_n(h_0) > \eta_0),
 \;\;
 \mathbb P_2= \sum_{k=1}^{\infty}\sum_{h_{k-1}\in T_{k-1}\atop
   h_{k}\in T_{k}}
 \!\! \mathbb P_\Delta(\nu_n(h_k-h_{k-1})> \eta_k ).\end{eqnarray*}
Then using Inequality (\ref{bernineg2}), we straightforwardly infer that
$\mathbb P_1\leq \exp(H_0-nx_0)$ and
$\mathbb P_2\leq
\sum_{k\geq 1}\exp(H_{k-1}+H_{k}-nx_{k})$ if we choose
\begin{eqnarray*}\Big\{\begin{array}{l} \eta_0=\sqrt{3x_0\|\alpha\|_{\infty,A}/2}
+\bar r_{(m,m')} x_0/3 \\
\eta_k=(1/2) \delta_{k-1}
(\sqrt{15 \|\alpha\|_{\infty,A} x_k}+  \bar r_{(m,m')} x_k).
\end{array}\Big. \end{eqnarray*}
Fix $u>0$ and choose $x_0$ such that
\begin{eqnarray*}
 nx_0 = H_0 +D_{m'} + u
\end{eqnarray*}
and for $k\geq 1$, $x_k$ such that
\begin{eqnarray*}
 nx_k =H_{k-1} + H_k + k D_{m'} + D_{m'} + u.
\end{eqnarray*}
If $D_{m'}\geq 1$, we infer that
\begin{eqnarray*}
 \mathbb P_\Delta\Big( \sup_{h\in  B_{m,m'}^{\mu}(0,1)} \nu_n(h) >
 \eta_0 + \sum_{k\geq 1} \eta_k \Big)
 \leq  e^{-D_{m'}-u} \Big(1 + \sum_{k=1}^{\infty} e^{-k D_{m'}} \Big)
 \leq  1.6 e^{-D_{m'}-u}.\end{eqnarray*}
Now, it remains to compute $\sum_{k\geq 0}\eta_k$. We note that $\sum_{k=0}
^{\infty} \delta_k = \sum_{k=0}^{\infty} k \delta_k = 2 \delta_0$.
This implies that:
\begin{align}
 \nonumber x_0 + &\sum_{k=1}^{\infty} \delta_{k-1}x_k \\
 \nonumber &\leq \Big[ \log(5/\delta_0)+ \delta_0 \sum_{k=1}^{\infty}
 2^{-(k-1)}[(2k-1)\log(2)+2\log(5/\delta_0)+k]\Big]\frac{D(m,m')}{n} \\
 \nonumber &+ \Big( 1+\delta_0\sum_{k\geq 1} 2^{-(k-1)}
 \Big)\frac{D_{m'}}{n}+\Big( 1+\delta_0\sum_{k\geq 1} 2^{-(k-1)}
 \Big)\frac{u}{n} \\
 & \leq \frac{a(\delta_0) D(m,m')}n + \frac{1+2\delta_0}{n} ( D_{m'}
 + u ),
 \label{intermed}
\end{align}
where $a(\delta_0) = \log(5 / \delta_0) + \delta_0( 4 \log(5 /
\delta_0) + 6\log(2) + 4)$. This leads to
\begin{eqnarray*}
\Big(\sum_{k=0}^{\infty} \eta_k
 \Big)^2 &\leq&  \frac{1}4
 \Big[\sqrt{2} \Big( \sqrt{3 \|\alpha\|_{\infty,A} x_0 / 2} + \bar
 r_{m,m'} x_0/3\Big) + \sum_{k=1}^{\infty} \delta_{k-1} \Big(
 \sqrt{15\|\alpha\|_{\infty,A}x_k} + \bar r_{m,m'} x_k\Big) \Big]^2
 \\ &\leq&  \frac 14 \Big[\Big(\sqrt{3\|\alpha\|_{\infty,A}x_0}+ \sum_{k=1}^{\infty}\delta_{k-1}\sqrt{15 \|\alpha\|_{\infty,A} x_k}\Big)
 + \bar r_{m,m'} \Big(\sqrt{2} x_0/3+ \sum_{k=1}^{\infty}\delta_{k-1}x_k\Big)\Big]^2\\
 &\leq&  \frac{15}4 \Big[\Big( \sqrt{x_0}+\sum_{k=1}^{\infty}
 \delta_{k-1} \sqrt{x_k} \Big)^2\|\alpha\|_{\infty,A} + \bar r_{m,m'}^2
 \Big(x_0+\sum_{k=0}^{\infty} \delta_{k-1}x_k\Big)^2 \Big] \\
 &\leq&  4 \Big[2 \Big(x_0+ \sum_{k=1}^{\infty}
 \delta_{k-1}x_k\Big)\|\alpha\|_{\infty,A} + \bar r_{m,m'}^2\Big(x_0+
 \sum_{k=1}^{\infty} \delta_{k-1}x_k\Big)^2\Big].
\end{eqnarray*}
Now, fix $\delta_0\leq  1/5$ (say,  $\delta_0=1/10$) and use the
bound (\ref{intermed}). The bound for $(\sum_{k=0}^{+\infty}
\eta_k)^2$ is less than a quantity proportional to:
\begin{eqnarray*}
 \Big( \frac{D(m,m')}n + \frac{D_{m'}}n
 \Big)\|\alpha\|_{\infty,A}+ \bar r_{m,m'}^2\Big(\frac{D(m,m')}n
 + \frac{D_{m'}}n \Big)^2 +\frac{\|\alpha\|_{\infty,A}u}n + \bar
 r_{m,m'}^2\frac{u^2}{n^2}.
\end{eqnarray*}
For collection [DP], we use that $\bar r_{m,m'}^2 \leq (r+1)^3 N_n /
(f_0 D(m,m'))$ and $N_n\leq n/\log n$ to obtain the bound:
\begin{eqnarray*}
 &&\bar r_{m,m'}^2 \Big(\frac{D(m,m')}n + \frac{D_{m'}}n \Big)^2
 \leq c(r+1)^3 \frac{N_n}{f_0 D(m,m')}\frac{D(m,m')^2}{n^2}
 \\ &\leq & \frac{c(r+1)^3}{f_0}\frac{N_nD(m,m')}{n^2} \leq
 \frac{c(r+1)^3}{f_0}\frac{1}{\log n}
 \frac{D(m,m')}n \leq \frac{D(m,m')}n.
\end{eqnarray*}
For collection [T], we have $\bar r_{m,m'}\leq C\sqrt{N_n}$ and
$N_n\leq \sqrt{n}/\log n$. We get
\begin{eqnarray*}
 \bar r_{m,m'}^2\Big(\frac{D(m,m')}n + \frac{D_{m'}}n \Big)^2
 &\leq & \frac{C N_n D(m,m')^2}{n^2} \leq \frac{C}{\log n}
 \frac{D(m,m')}{n} \leq \frac{D(m,m')}{n}.
\end{eqnarray*}
Thus, for both the cases, the bound for $(\sum \eta_k)^2$ is
proportional to:
\begin{eqnarray*}
 (1+\|\alpha\|_{\infty,A})\Big[\frac{D(m,m')}n +
 \frac{D_{m'}}n \Big]+\frac{\|\alpha\|_{\infty,A}u}n + \bar
 r_{m,m'}^2\frac{u^2}{n^2}.
\end{eqnarray*}
We obtain, as $D(m,m') \leq D_m + D_{m'}$,
\begin{eqnarray*}
 & & \mathbb P_\Delta\Big[\sup_{h \in
   B_{m,m'}^{\mu }(0,1)}[\nu_n(h)]^2> \kappa
 \Big( (1 + \|\alpha\|_{\infty,A}) \frac{D_m+ D_{m'}}n +
 (\frac{\| \alpha \|_{\infty,A} u} n\vee \bar
 r_{m,m'}^2 \frac{u^2}{n^2}) \Big) \Big]\\ & \leq &
 \mathbb P_\Delta\Big[\sup_{h\in  B_{m,m'}^{\mu
   }(0,1)}[\nu_n(h)]^2 > \eta^2 \Big] \leq
 2\; \mathbb P_\Delta\Big[\sup_{h\in  B_{m,m'}^{\mu
   }(0,1)}\nu_n(h)> \eta \Big] \leq  3.2
 e^{-D_{m'}-u}
\end{eqnarray*}
so that, if we take $\kappa_\alpha := \kappa (1 +
\norm{\alpha}_{\infty, A})$,
\begin{align*}
 {\mathbb E} \Big[\Big(\sup_{h\in B_{m,m'}^{\mu}(0,1)} &\nu_n^2(h) -
 p(m, m') \Big)_+ \1(\Delta) \Big]\\
 &\leq \int_0^{\infty}{\mathbb P}_\Delta \Big(\sup_{h\in
   B_{m,m'}^{\mu}(0,1)} \nu_n^2(h) > p(m, m') + u\Big)du \\
 &\leq e^{ -D_{m'} } \Big( \int_{2 \kappa_\alpha / \bar r_{m,
     m'}^2}^{\infty} e^{-n u / (2\kappa_\alpha )} du +
 \int_0^{2\kappa_\alpha / \bar r_{m,m'}^2} e^{-n \sqrt{u} / (2
   \sqrt{\kappa_\alpha} \bar r_{m,m'} ) }du\Big) \\
 &\leq e^{-D_{m'}}
 \frac{2\kappa_\alpha}n\Big(\int_0^{\infty}e^{-v}dv
 +\frac{2\bar r_{m,m'}^2}n\int_0^{\infty}e^{-\sqrt{v}}dv\Big)\\
 &\leq e^{-D_{m'}} \frac{2\kappa_{\alpha }}n(1+ \frac{4\bar
   r_{m,m'}^2}n) \leq \frac{\kappa'_\alpha e^{-D_{m'}}}{n},
\end{align*}
where $\kappa_\alpha'$ is a constant depending on
$\norm{\alpha}_{\infty, A}$. This ends the proof of
Lemma~\ref{propsup}. \\
To conclude the proof of Proposition \ref{tala}, we just have to
bound $\sum_{m'\in {\mathcal M}_n}e^{-D_{m'}}$. This term is at most
\begin{eqnarray*}
\sum_{j, k\geq 1}e^{-jk}&=&\sum_{j=1}^{\infty} \sum_{k=1}^{\infty}
(e^{-j})^k=\sum_{j=1}^{\infty}\frac{e^{-j}}{1-e^{-j}} \leq  \frac
1{1-e^{-1}} \sum_{j=1}^{\infty}e^{-j}=\frac{e^{-1}}{(1-e^{-1})^2}.
\end{eqnarray*}

\end{proof}

\section{Proof of the auxiliary results}
\label{auxiliary}

\subsection{Proof of Proposition \ref{Omegacomp}}
Let $\hat  f_{m_1^*}$ and $\hat f_0$ be defined by \eqref{condinf}, with $m^*_1=(D_{m_1}, \mathcal D_n^{(2)})$ with $\log n\leq
D_{m_1} \leq n^{1/4}/ \sqrt{\log n}$ and $\mathcal D_n^{(2)}\leq n^{1/4}/ \sqrt{\log n}$, see $(\mathcal M_1)$.
We remark that, for all $(x,z)\in \mathbb R^2$,
$$\hat f_{m_1^*}(x,z)=f(x,z)+\hat f_{m_1^*}(x,z)-f(x,z) \geq f_0-\|\hat f_{m_1^*}-f\|_{\infty,A}.$$
We deduce that $\|\hat f_{m_1^*}-f\|_{\infty,A}\geq  f_0- \hat f_0$. In the same manner, $\|\hat f_{m_1^*}-f\|_{\infty,A}\geq \hat f_0-f_0$. Thus
$${\mathbb P}(\Omega^\complement)={\mathbb P}(|f_0-\hat f_0|>f_0/2) \leq {\mathbb P}(\|\hat f_{m_1^*}-f\|_{\infty,A}>f_0/2).$$
Therefore, we just have to prove that ${\mathbb P}(\|\hat f_{m_1^*}-f\|_{\infty,A}>f_0/2)\leq C_k/n^k$.

First remark that $\|\hat f_{m_1^*}-f\|_{\infty,A}\leq \|\hat
f_{m_1^*}-f_{m_1^*}\|_{\infty,A}+\|f_{m_1^*}-f\|_{\infty,A}.$ As
$f\in B_{2,\infty}^{(\tilde \beta_1,\tilde \beta_2)}(A)$ with $\bar{\tilde  \beta}>1$, the
imbedding theorem proved in \cite{NIK} p.236 implies that $f$
belongs to $B_{\infty,\infty}^{(\beta_1^*,\beta_2^*)}(A)$ with
$\beta_1^*=\tilde\beta_1(1-1/\bar{\tilde\beta})$ and
$\beta_2^*=\tilde\beta_2(1-1/\bar{\tilde\beta})$. Then the approximation lemma of
\cite{LAC} recalled in Section \ref{coro},  which is still valid for the trigonometric polynomial
spaces with the infinite norm instead of the $L^2$ norm, yields to
$$\|f_{m_1*}-f\|_{\infty,A}\leq C(D_{m_1*}^{-\beta_1^*} + ( \mathcal D_n^{(2)})^{-\beta_2^*}).$$
As we assumed that $D_{m_1^*}\geq \log n$, it follows that $\|f_{m_1*}-f\|_{\infty,A}$ tends to zero when $n \to +\infty$. Thus, for $n$ large enough, we have $\|f_{m_1*}-f\|_{\infty,A}\leq f_0/4$ and \begin{eqnarray*}
{\mathbb P}(\|\hat f_{m_1^*}-f\|_{\infty,A}>f_0/2)\leq {\mathbb P}(\|\hat f_{m_1^*}-f_{m_1^*}\|_{\infty,A}>f_0/4).
\end{eqnarray*}
Now, following $(\mathcal M2)$, we get \begin{eqnarray*}
\|\hat f_{m_1^*}-f_{m_1^*}\|_{\infty,A} \leq \sqrt{\phi_1\phi_2 D_{m_1^*}  \mathcal D_n^{(2)}}
\|\hat f_{m_1^*}-f_{m_1^*}\|.
\end{eqnarray*}
Now we define
\begin{equation}\label{proc} \displaystyle\vartheta_n(h)=\frac{1}{n}\sum_{i=1}^{n}\int\Big( h(X_i,y)Y^i(y)-{\mathbb
E}\big(h(X_i,y)Y^i(y)\big)\Big)dy=\|\sqrt h\|^2_n-\|\sqrt h\|^2_{\mu}.\end{equation}
With this notation, and reminding of ({\ref{eq:repest2}) and of the proof of Proposition \ref{riskf} in Section \ref{sec:estimation_de_f}, we have
$$\|\hat f_{m_1^*}-f_{m_1^*}\|^2= \sum_{j,k}(\hat b_{j,k}-b_{j,k})^2=\sum_{j,k} \vartheta_n^2(\varphi_j^{m^*_1}\otimes \psi_k^{m^*_1}).$$
Thus
\begin{eqnarray*} {\mathbb P}(\|\hat f_{m_1^*}-f\|_{\infty,A}>f_0/2)&\leq & {\mathbb P}\Big(\sum_{j,k} \vartheta^2_n(\varphi_j^{m^*_1}\otimes \psi_k^{m^*_1})\geq \frac{f_0^2 }{16\phi_1\phi_2 D_{m_1^*}  \mathcal D_n^{(2)}}\Big) \\&\leq & \sum_{j,k}{\mathbb P}\Big( \vartheta^2_n(\varphi_j^{m^*_1}\otimes \psi_k^{m^*_1})\geq \frac{f_0^2 }{16\phi_1\phi_2 (D_{m_1^*}  \mathcal D_n^{(2)})^2}\Big) \\ &\leq &
\sum_{j,k} {\mathbb P}\Big(|\vartheta_n(\varphi_j^{m^*_1}\otimes \psi_k^{m^*_1})|\geq \frac{f_0 }{4\sqrt{\phi_1\phi_2}D_{m_1^*}  \mathcal D_n^{(2)}}\Big).
\end{eqnarray*}
Notice that $\vartheta_n(\varphi_j^{m^*_1}\otimes \psi_k^{m^*_1})=\frac1n \sum_1^n ( U_i^{j,k}-\mathbb E(U_i^{j,k}) )$,
where $U_i^{j,k} = \varphi_j(X_i)\int \psi_k(y)Y^i(y)dy$ are i.i.d. r.v. We can apply the Bernstein inequality to $\vartheta_n$ i.e. to the i.i.d. r.v. $U_i^{j,k}$. Indeed, we have
$$
\|U_i^{j,k}\|_{\infty} \leq \|\varphi_j\|_{\infty} \int |\psi_k(y)|dy \leq \|\varphi_j\|_{\infty} (\int
\psi_k^2(y)dy)^{1/2} \leq   \sqrt{  \phi_1 D_{m_1^*} }:=c$$
and  $\mathbb E[(U_i^{j,k})^2] \leq \|f_X\|_{\infty,A}=v^2$.
We get $${\mathbb P}\Big(|\vartheta_n(\varphi_j^{m^*_1}\otimes \psi_k^{m^*_1})|\geq \frac{f_0
}{4\sqrt{\phi_1\phi_2}D_{m_1^*}  \mathcal D_n^{(2)}}\Big)\leq 2 \exp(-\frac{nx^2/2}{v^2+cx})$$
with $x=f_0/(4\sqrt{\phi_1\phi_2}D_{m_1^*}  \mathcal D_n^{(2)})$ and $v$ and $c$ are right above. That is:
$${\mathbb P}\Big(|\vartheta_n(\varphi_j^{m^*_1}\otimes \psi_k^{m^*_1})|\geq \frac{f_0 }{4\sqrt{\phi_1\phi_2}D_{m_1^*}  \mathcal D_n^{(2)}}\Big)\leq 2 \exp\Big(- \frac{C n f_0^2 }{16 \phi_1 \phi_2 (D_{m_1^*}  \mathcal D_n^{(2)})^{2}} \Big).$$ As both $D_{m_1^*}$ and ${\mathcal D}_n^{(2)}$ are less than $n^{1/4}/\sqrt{\log(n)}$, we obtain:
\begin{eqnarray*}
{\mathbb P}(\Omega^\complement)\leq 2 D_{m_1^*}  \mathcal D_n^{(2)}\exp\Big(- \frac{C n f_0^2 }{16 \phi_1 \phi_2 (D_{m_1^*}  \mathcal D_n^{(2)})^{2}} \Big)\leq 2 \sqrt{n}
\exp\Big(- C'(\log n)^2\Big)\leq \frac{C'_k}{n^k},
\end{eqnarray*}for any $k$ arbitrarily large, when $n$ is large
enough.

\subsection*{Proof of Proposition~\ref{alpha4}}

Note that $\hat\alpha_{\hat m}$ is either $0$ or $\argmin_{t\in
S_{\hat m}}\gamma_n(t)$. Let us denote for short $\varphi_{j} :=
\varphi_{j}^{\hat m}$ and $\psi_k := \psi_k^{\hat m}$. In the second
case, $\min{\rm Sp}(G_{\hat m})\geq \max (\hat f_0 / 3, n^{-1/2})$
and thus \begin{eqnarray*}
\|\hat\alpha_{\hat m}\|^2 &=& \sum_{j,k}(\hat a^{\hat m}_{j,k})^2 =
\|A_{\hat m}\|^2 = \|G_{\hat m}^{-1} \Upsilon_{\hat m}\|^2 \\
&\leq & (\min{\rm Sp}(G_{\hat m}))^{-2} \| \Upsilon_{\hat m} \|^2
\leq \min( 9 / \hat f_0^2, n) \sum_{j,k} \Big( \frac 1n \sum_{i=1}^n
\varphi_j(X_i) \int \psi_k(z) dN^i(z) \Big)^2
\\ &\leq & \min(9/\hat f_0^2, n) \frac 1n \sum_{i=1}^n
\sum_{j}\varphi_j^2(X_i) \sum_{k}\Big(\int \psi_k(z)dN^i(z)\Big)^2
\\ &\leq & \min(9/\hat f_0^2, n)\phi_1{\mathcal
  D}_n^{(1)}\frac 1n\sum_{i=1}^n \sum_{k} \Big(\int
  \psi_k(z)dN^i(z)\Big)^2.
\end{eqnarray*}
Therefore,
\begin{eqnarray}\nonumber
\|\hat\alpha_{\hat m}\|^4 &\leq & n^2\phi_1^2({\mathcal
D}_n^{(1)})^2 \frac 1n\sum_{i=1}^n \left(\sum_{k} \Big(\int
\psi_k(z)dN^i(z)\Big)^2\right)^2 \\ \label{compil1} &\leq & n^2\phi_1^2({\mathcal
D}_n^{(1)})^2 {\mathcal D}_n^{(2)} \frac 1n\sum_{i=1}^n \sum_{k}
\Big(\int \psi_k(z)dN^i(z)\Big)^4 .
\end{eqnarray} Now, we have:
\begin{eqnarray*}
&&\mathbb E \Big (\frac 1n\sum_{i=1}^n \sum_{k} \Big(\int
\psi_k(z)dN^i(z)\Big)^4\Big) \\& \leq& 2^3 \frac 1n\sum_{i=1}^n
\sum_{k} \mathbb E \Big (\Big(\int
\psi_k(z)dM^i(z)\Big)^4\Big)+2^3 \frac 1n\sum_{i=1}^n
\sum_{k}\mathbb E \Big (\Big(\int \psi_k(z)
\alpha(X^i,z)Y^i(z)dz\Big)^4\Big).
\end{eqnarray*}
Using the B\"{u}rkholder Inequality as recalled in
\cite{lipstershiryayev} p~75, and the fact that the quadratic variation process of each $M^i$ is $N^i$ ($i=1,\dots,n$), we obtain:
\begin{eqnarray*}
&&\mathbb E \Big (\frac 1n\sum_{i=1}^n \sum_{k} \Big(\int \psi_k(z)dN^i(z)\Big)^4\Big)\\& \leq& 2^3 C_b \frac
1n\sum_{i=1}^n \sum_{k} \mathbb E \Big (\Big(\int
\psi_k^2(z)dN^i(z)\Big)^2\Big)+2^3 \frac 1n\sum_{i=1}^n \sum_{k}
\mathbb E \Big (\Big(\int \psi_k(z)
\alpha(X^i,z)Y^i(z)dz\Big)^4\Big) \\&\leq& 2^3C_b \frac
1n\sum_{i=1}^n \sum_{k} \mathbb E \Big (\Big(\sum_{s: \Delta
N^i(s)\neq 0} \psi_k^4(s)\Big)\Big)+2^3 \frac 1n\sum_{i=1}^n
\sum_{k} \mathbb E \Big (\Big(\int \psi_k(z)
\alpha(X^i,z)Y^i(z)dz\Big)^4\Big)\\& \leq& 2^3 C_b\frac
1n\sum_{i=1}^n  \mathbb E \Big (\Big(\sum_{s: \Delta N^i(s)\neq 0}
\sum_{k} \psi_k^4(s)\Big)\Big)+2^3 \frac 1n\sum_{i=1}^n \sum_{k}
\mathbb E \Big (\Big(\int \psi_k(z)
\alpha(X^i,z)Y^i(z)dz\Big)^4\Big)\\& \leq& 2^3 C_b \phi_2 (\mathcal D_n^{(2)})^2 \frac
1n\sum_{i=1}^n  \mathbb E \Big (\Big(\sum_{s: \Delta N^i(s)\neq
0}1\Big)\Big)+2^3 \frac 1n\sum_{i=1}^n \sum_{k} \mathbb E \Big
(\Big(\int \psi_k(z) \alpha(X^i,z)Y^i(z)dz\Big)^4\Big)
\\&\leq& 2^3 C_b\phi_2 (\mathcal D_n^{(2)})^2 \frac 1n\sum_{i=1}^n  \mathbb E ( N^i(1))+2^3 \frac
1n\sum_{i=1}^n \sum_{k} \mathbb E \Big (\Big(\int \psi_k(z)
\alpha(X^i,z)Y^i(z)dz\Big)^4\Big)
\end{eqnarray*}
This yields, using Assumptions $({\mathcal A}3)$ and $({\mathcal
A}4)$:
\begin{eqnarray} \nonumber
&&\mathbb E \Big (\frac 1n\sum_{i=1}^n \sum_{k} \Big(\int
  \psi_k(z)dN^i(z)\Big)^4\Big) \\ \nonumber &\leq& C\Big( \phi_2 (\mathcal D_n^{(2)})^2 \mathbb E (
  N^1(1)) + \sum_{k} \mathbb E \Big
  (\Big(\int \psi_k(z) \alpha(X, z) Y(z)dz\Big)^4\Big) \Big) \\ \nonumber
&\leq & C\Big( \phi_2 (\mathcal D_n^{(2)})^2\mathbb E ( N^1(1))+ \|\alpha\|_{\infty, A}^4
  \|\sum_k \psi_k^2\|_{\infty, A} \sum_k \int \psi_k^2(z)dz\Big) \\
\label{compil2} &\leq & C\Big(\phi_2 (\mathcal D_n^{(2)})^2 \mathbb
  E ( N^1(1))+ \|\alpha\|_{\infty, A}^4\phi_2({\mathcal
    D}_n^{(2)})^2\Big).
\end{eqnarray}
Then we have, by inserting (\ref{compil2}) in (\ref{compil1}),  \begin{eqnarray*} {\mathbb E}(\|\hat \alpha_{\hat
m}\|^4)&\leq & (\phi_1 n{\mathcal D}_n^{(1)})^2 {\mathcal
D}_n^{(2)}\mathbb E \Big (\frac 1n\sum_{i=1}^n \sum_{k} \Big(\int
\psi_k(z)dN^i(z)\Big)^4\Big) \\ &\leq & Cn^2({\mathcal D}_n^{(1)})^2
({\mathcal D}_n^{(2)})^3 \leq C' n^{4.5}\leq C'
n^{5},\end{eqnarray*} as we claim that we can reach ${\mathcal
D}_n^{(i)}\leq \sqrt{n}/\log(n)$ in the case of localized bases
[DP], [W], [H]. Note that for basis [T], under $({\mathcal M}1)$, the final order is much less (namely $n^{3.25}$ instead of $n^{4.5}$).

\subsection*{Proof of Proposition~\ref{deltacomp}}

Define, for $\rho> 1$, the set \begin{eqnarray*}\Delta_{\rho}=\{\forall h\in {\mathcal S_n},
\Big|\|h\|_n^2/\|h\|^2_{\mu}-1\Big| \leq 1-
1/\rho \},\end{eqnarray*} where $\mathcal S_n$ is the set of maximal dimension
of the collection. Remark that $\Delta=\Delta_{2}$, see \eqref{deltaset}. First we observe that:
\begin{eqnarray*}
{\mathbb P}(\Delta_\rho^\complement)\leq
\mathbb P\Big(\sup_{h\in B_{\mathcal S_n}^{\mu}(0,1)}
|\vartheta_n(h^2)|>1-1/\rho\Big)\end{eqnarray*} where $\vartheta_n(\cdot)$ is defined by (\ref{proc})
and $B_{\mathcal S_n}^{\mu}(0,1)=\{t \in \mathcal{S}_n, \|t\|_{\mu} \leq 1\}.$ We denote by $(\varphi_j\otimes \psi_k)$ the
$\mathbb L^2$-orthonormal basis of ${\mathcal S}_n$.
If $h(x,y)=\sum_{j,k}a_{j,k}\varphi_j(x)\psi_k(y)$, then
\begin{equation}\label{vartheta} \vartheta_n(h^2)= \sum_{j,k,j',k'}a_{j,k}a_{j',k'} \vartheta_n((\varphi_j
\otimes \psi_k)(\varphi_{j'}\otimes \psi_{k'})).\end{equation} We obtain
\begin{equation}\label{sup2}\sup_{h\in B_{\mathcal S_n}^{\mu}(0,1)}
|\vartheta_n(h^2)|\leq f_0^{-1}\sup_{\sum a_{j,k}^2\leq 1}\Big|
\sum_{j,k,j',k'}a_{j,k}a_{j',k'} \vartheta_n((\varphi_j
\otimes \psi_k)(\varphi_{j'}\otimes \psi_{k'}))\Big|.
\end{equation}

\begin{lemn}[\cite{BCV}]
Let $B_{j,j'}=\|\varphi_j\varphi_{j'}\|_{\infty,A}$ and $V_{j,j'}=\|\varphi_{j}\varphi_{j'}\|_2$.
Let, for any symmetric matrix $(A_{j,j'})$
$$ \bar\rho(A):=\sup_{\sum b_j^2\leq 1} \sum_{j,j'}|b_jb_{j'}|A_{j,j'}$$
and $L(\varphi):=\max\{\bar\rho^2(V),\bar\rho(B)\}.$
Then, if $(\mathcal M 2)$ is satisfied, we have $ L(\varphi)\leq \phi_1(\mathcal{D}_n^{(1)})^2$,
and $ L(\varphi)\leq 5\phi_1^4 \mathcal{D}_n^{(1)}$, if the basis is localized (cases [P] or [W]).
\end{lemn}

Let us define \begin{eqnarray*}x&:=&\cfrac{f_0^2(1-1/\rho)^2}{4\|f_X\|_{\infty,A} ({\mathcal D}_n^{(2)})^2 L(\varphi)}
\text{ and }\\
\Theta &:=&\Big\{\forall (j,k) \forall (j',k')
\quad|\vartheta_n((\varphi_j
\otimes \psi_k)(\varphi_{j'}\otimes \psi_{k'}))|\leq
      4\Big(B_{j,j'} x+V_{j,j'}\sqrt{2\|f_X\|_{\infty,A} x}\Big)\Big\}. \end{eqnarray*}
Starting from (\ref{sup2}), we have, on $\Theta$:
\begin{eqnarray*}
\sup_{h\in B_{\mathcal S_n}^{\mu}(0,1)} |\vartheta_n(h^2)|
\leq 4f_0^{-1}\sup_{\sum a_{j,k}^2\leq 1}\sum_{j,j'}(\sum_{k,k'}
|a_{j,k}a_{j',k'}|)\Big(B_{j,j'} x +V_{j,j'}\sqrt{2\|f_X\|_{\infty,A}
x}\Big).
\end{eqnarray*}
Thus setting $b_j=\sum_k |a_{j,k}|$, we have $\sum_j b_j^2 \leq
{\mathcal D}_n^{(2)}$ and it follows that, on $\Theta$,
\begin{eqnarray*}
\sup_{h\in B_{\mathcal S_n}^{\mu}(0,1)} |\vartheta_n(h^2)|
&\leq&  f_0^{-1} {\mathcal D}_n^{(2)} \sup_{\sum b_j^2=1}\sum_{j,j'} |b_jb_{j'}|
\Big(B_{j,j'} x+V_{j,j'}\sqrt{2\|f_X\|_{\infty,A} x}\Big)\\
&\leq&  f_0^{-1}{\mathcal D}_n^{(2)} \Big(\bar\rho(B) x+\bar\rho(V)\sqrt{2\|f_X\|_{\infty,A} x}\Big)\\
&\leq&(1-1/\rho)\Big(\frac{f_0(1-1/\rho)}{4{\mathcal D}_n^{(2)}\|f\|_{\infty,A}}\frac{\bar\rho(B)}{L(\varphi)}
+\frac{1}{\sqrt{2}}\Big(\frac{\bar\rho^2(V)}{L(\varphi)}\Big)^{1/2}\Big)\\
&\leq&(1-1/\rho)\Big(\frac{1}{4}+\frac{1}{\sqrt{2}}\Big)\leq(1-1/\rho).
\end{eqnarray*}
Therefore, \begin{eqnarray*} {\mathbb P}\Big(\sup_{t\in B_{\mathcal S_n}^{\mu}(0,1)} |\vartheta_n(t^2)|>1-\frac{1}{\rho}
\Big)\leq {\mathbb P}(\Theta^\complement).\end{eqnarray*}

Let $\phi_{\lambda}=\varphi_j\otimes \psi_k$ for $\lambda=(j,k)$. To bound ${\mathbb P}(\vartheta_n(\phi_{\lambda}\phi_{\lambda'})\geq
B_{j,j'} x+V_{j,j'}\sqrt{2\|f_X\|_{\infty,A} x})$, we will apply the
Bernstein inequality given in \cite{BM8} to the i.i.d.
r.v.
\begin{eqnarray}\label{eqn:Ui}
U_i^{\lambda,\lambda'}= U_i^{(j,k),(j',k')} =
\varphi_j(X_i)\varphi_{j'}(X_i)\int \psi_k(y)\psi_{k'}(y)Y^i(y)dy.\end{eqnarray}
Under $(\mathcal A4)$, the r.v. are bounded
$$|U_i^{\lambda,\lambda'}|\leq \|\varphi_j\varphi_{j'}\|_{\infty,A}
\int |\psi_k(y)\psi_{k'}(y)|dy\leq
\|\varphi_j\varphi_{j'}\|_{\infty,A}=B_{j,j'}.$$  Moreover, using $({\mathcal A}4)$ again, we obtain:
$$(U_i^{\lambda,\lambda'})^2 \leq (\varphi_j(X_i)\varphi_{j'}(X_i))^2 \int \psi_k^2(y)dy \int \psi^2_{k'}(y)dy= (\varphi_j(X_i)\varphi_{j'}(X_i))^2 $$ and thus
$${\mathbb E}[(U_i^{\lambda,\lambda'})^2] \leq {\mathbb E}[(\varphi_j(X_i)\varphi_{j'}(X_i))^2]\leq \|f_X\|_{\infty,A} V_{j,j'}^2.$$
We get
$${\mathbb P}(|\vartheta_n(\phi_{\lambda}\phi_{\lambda'})|\geq B_{j,j'} x+V_{j,j'}\sqrt{2\|f_X\|_{\infty,A} x})
\leq 2e^{-nx}.$$
Given that ${\mathbb P}(\Delta_\rho^\complement)\leq {\mathbb P}(\Theta^\complement)=
\sum_{\lambda,\lambda'} {\mathbb P}\Big(|\vartheta_n(\phi_{\lambda}\phi_{\lambda'})|>B_{j,j'}
x+V_{j,j'}\sqrt{2\|f_X\|_{\infty,A} x}\Big)$, we can write:
\begin{eqnarray*}
{\mathbb P}(\Delta_\rho^\complement)&\leq&2(\mathcal{D}_n^{(1)} {\mathcal
D}_n^{(2)})^2\exp\Big\{-\cfrac{nf_0^2(1-1/\rho)^2}
{4\|f_X\|_{\infty,A} ( {\mathcal D}_n^{(2)})^2 L(\varphi)}\Big\}\\
&\leq  & 2n^2\exp\Big\{-\cfrac{f_0^2(1-1/\rho)^2}{4\|f_X\|_{\infty,A}}
\cfrac{n}{({\mathcal D}_n^{(2)})^2L(\varphi)}\Big\}.
\end{eqnarray*}
Following the lemma of ~\cite{BCV} above, and using
Assumption~$(\mathcal M_1)$, we have
\begin{eqnarray*}
({\mathcal D}_n^{(2)})^2
L(\varphi)\leq\phi_1({\mathcal
  D}_n^{(2)}\mathcal{D}_n^{(1)})^2\leq\phi_1
n/\log^2(n).
\end{eqnarray*}
And then, we have for any $k$ arbitrarily large, when $n$ is large
enough,
\begin{equation}
\label{eqom}
{\mathbb P}(\Delta_\rho^\complement)\leq
2n^2\exp\Big\{-\cfrac{f_0^2(1-1/\rho)^2}{40\|f\|_{\infty,A}\phi_1}
\log^2(n)\Big\} \leq \frac{C_k}{n^k}.
\end{equation}
Now, if the basis is localized, the result is better. In this case,
$L(\varphi)\leq 5\phi_1^4{\mathcal D}_n^{(1)}$. Moreover, take
histogram basis in (\ref{vartheta}), then all terms with $k\neq k'$
vanish and then we can take $b_j=(\sum_k a_{j,k}^2)^{1/2}$
directly. Then, as then $\sum_j b_j^2\leq 1$, we obtain
\begin{eqnarray*}
{\mathbb P}(\Delta_\rho^\complement)&\leq&2(\mathcal{D}_n^{(1)})^2 {\mathcal
D}_n^{(2)}\exp\Big\{-\cfrac{nf_0^2(1-1/\rho)^2} {40\|f_X\|_{\infty,A}
L(\varphi)}\Big\} \leq
2n^2\exp\Big\{-\cfrac{f_0^2(1-1/\rho)^2}{40\|f_X\|_{\infty,A}}
\cfrac{n}{L(\varphi)}\Big\}.
\end{eqnarray*}
Thus  $L(\varphi)\leq 5\phi_1^4 \mathcal{D}_n^{(1)}\leq\phi_1
n/\log^2(n)$ is enough to get (\ref{eqom}) again. The proof is easy
to extend to any localized basis as $[P]$ or $[W]$, (with ${\mathcal
D}_n^{(2)}$ in the bound of $\sum_j b_j^2$ replaced by $r+1$ in case $[P]$ for instance).

\subsection*{Proof of Lemma \ref{inclus}}

Let $m \in \mathcal M_n$ be fixed and let $\ell$ be an eigenvalue of
$G_m$. There exists $A_m \neq 0$ with coefficients
$(a_{\lambda})_{\lambda}$ such that $G_m A_m = \ell A_m$ and thus
$A_m^\top G_m A_m = \ell A_m^\top A_m$. Now, take $h := \sum_{\lambda}
a_{\lambda} \varphi_{\lambda} \in S_m$. We have $\|h\|_n^2 =
A_m^{\top} G_m A_m$ and $\|h\|_A^2 = A_m^{\top} A_m$. Thus, on
$\Delta$ (see~\eqref{deltaset}):
\begin{equation*}
A_m^{\top} G_m A_m = \|h\|_n^2 \geq \frac 12 \| h \|_{\mu}^2\geq \frac
12 f_0\| h \|^2_A = \frac 12 f_0 A_m^{\top} A_m.
\end{equation*}
Therefore, on $\Delta$, for all $m \in \mathcal M_n$, we have $\min
\spec (G_m) \geq f_0 / 2$. Moreover, on $\Omega$, we have $f_0 \geq 2
\hat f_0 / 3$ and $\max( \hat f_0 / 3, n^{-1/2}) = \hat f_0$, for
$n\geq 36/f_0^2$. \qed

\end{document}